\providecommand{\pgfsyspdfmark}[3]{}
\documentclass[journal, web]{article} 
\usepackage{color}
\usepackage{cite}
\usepackage{amsmath,amssymb,amsfonts}  
\usepackage{graphicx,wrapfig}

  \usepackage[keeplastbox]{flushend}
  \allowdisplaybreaks
  \makeatletter
   \let\NAT@parse\undefined
  \makeatother
 \usepackage[pdftex, pdfborderstyle={/S/U/W 0}]{hyperref}

\usepackage{comment}

 \usepackage{enumitem}

 \newtheorem{thm}{Theorem}
 \newtheorem{exe}{Example}
 \newtheorem{rem}{Remark}
 \newtheorem{lem}{Lemma}
 \newtheorem{prop}{Proposition}
 
 \newtheorem{defin}{Definition}
 \newtheorem{ass}{Assumption}

 \newenvironment{theorem}[1][\,\!]{\ignorespaces\begin{thm}[#1]\rm }{\hfill \mbox{\footnotesize$\square$} \end{thm}}
 \newenvironment{example}{\ignorespaces\begin{exe}\rm }{\hfill \mbox{\footnotesize$\square$} \end{exe}}
 
 \newenvironment{proposition}{\begin{prop}}{\hfill \mbox{\footnotesize$\square$} \end{prop}}
 \newenvironment{remark}{\begin{rem}}{\hfill $\bullet$ \end{rem}}
 \newenvironment{lemma}{\begin{lem}}{\hfill  \mbox{\footnotesize$\square$} \end{lem}}
 \newenvironment{definition}{\vskip 3pt\begin{defin}}{\vskip 3pt
                                          \end{defin}}
 \newenvironment{assumption}{\vskip 3pt\begin{ass}}{\end{ass}\vskip 3pt}
 \newenvironment{proof}{\it Proof. }{\hfill \mbox{\footnotesize$\blacksquare$}}

\def\appbkgd{I} 
\def\Lemspecres{Lemma \ref{Lemspecres}\color{black}}
\def\SectionAH{Section \ref{A-H}\color{black}}
\def\lemAE{Lemma \ref{lemAE} in Appendix \appbkgd\color{black}}
\def\lemstabmani{Lemma \ref{lemstabmani}\color{black}}
\def\thmtykho{ Lemma \ref{thmtykho}\color{black}}
\def\RemThyk{ Remark \ref{RemThyk}\color{black}}
\def\lemregpert{Lemma \ref{lemregpert}\color{black}}
\def\lemcontin{Lemma \ref{lemcontin}\color{black}}
\def\lemanal{Lemma \ref{lemanal}\color{black}}

\newcommand{\sfrac}[2]{\mbox{\small $\displaystyle\frac{#1}{#2}$}}

\title{{\bf Singular-Perturbations-Based Analysis of Dynamic Consensus in Directed Networks of Heterogeneous Nonlinear Systems }}

\author{Mohamed Maghenem \quad Elena Panteley \quad Antonio Lor\'{\i}a 
  \thanks{The authors are with CNRS, France.  M. Maghenem is with University of Grenoble Alpes,  CNRS,  Grenoble INP,  GIPSA-Lab,  France. 
        E-mail: mohamed.maghenem@cnrs.fr; 
        E. Panteley and A. Lor\'ia are with L2S, CNRS, 91192 Gif-sur-Yvette, France. 
        E-mail: elena.panteley@cnrs.fr and antonio.loria@cnrs.fr 
        This work was supported by the French ANR via project HANDY, contract number 
        ANR-18-CE40-0010.
        }}

\begin{document}
\maketitle 

\begin{abstract}
We analyze networked heterogeneous nonlinear systems, with diffusive coupling and interconnected over a generic static directed graph.  Due to the network's hetereogeneity,  complete synchronization is impossible, in general, but  an emergent dynamics arises. This may be characterized by two dynamical systems evolving in two time-scales. The first, ``slow'',  corresponds to the dynamics of the network on the synchronization manifold.  The second, ``fast'', corresponds to that of the synchronization errors.   We present a framework to analyse the emergent dynamics based on the behavior of the slow dynamics.  Firstly,  we give conditions under which  if the slow dynamics admits a globally asymptotically stable equilibrium,  so does the networked systems.  Secondly,  we give conditions under which,  if the slow dynamics admits an asymptotically stable orbit and a single unstable equilibrium point, there exists a unique periodic orbit that is 
almost-globally asymptotically stable.  The emergent behavior is thus clear,  the systems asymptotically synchronize in frequency and,  in the limit,  as the coupling strength grows,  the emergent dynamics approaches that of the slow system. Our analysis is established using singular-perturbations theory.  In that regard,  we contribute with original statements on stability of disconnected invariant sets and limit cycles.
 \\[3pt]
 Keywords---Consensus, multi-agent systems, singular perturbations, network systems, synchronization. \\[3pt]
Mathematics Subject Classification---34D15, 34D05, 34C25, 93C10, 34C28
\end{abstract}

\section{Introduction and motivation}

Networks of nonlinear heterogeneous systems are both, ubiquitous and {\it complex}.  Their ubiquity motivates their study across  numerous research disciplines,  as varied as Engineering Systems theory \cite{chow1982}, Complexity theory \cite{heylighen1989self} or,  even,  Philosophy of Science \cite{Silberstein2002-SILREA}.  Their complexity is motor for two apparently antagonistic trains of thought that attempt to explain the collective behavior of networked systems in a broad sense: {\it reductionism} and {\it emergentism}. The first asserts that any whole can be reduced to its constituent parts---as in the case of networked linear systems \cite{WEIBOOK},  while the tenet of emergentism is that a new behavior appears as a consequence of the interaction of the said parts 
\cite{Silberstein2002-SILREA}---as in networks of heterogeneous nonlinear systems.  What is more,  one of the accepted definitions of Complexity is that it corresponds to the difference between the network as a whole and the sum of its parts and,  in that regard, nonlinearity is a necessary condition for complexity to appear \cite{Arecchi1997}.  

In this paper,  we show that, to some extent, both emergentism and reductionism are not necessarily mutually exclusive,  but their respective underlying postulates are useful to assess the behavior of networks of heterogeneous nonlinear systems. We focus on systems with dynamics given by 
\begin{equation} \label{dxifi}
   \dot{x}_i = f_i(x_i) + u_i, \quad i \in \{1,2,\ldots,N\}, \ x_i \in \mathbb{R}^n,
\end{equation}
where $i \in \{1,2,\ldots,N\}$, $x_i \in \mathbb{R}^n$ is the state of the $i$th system and  $u_i \in \mathbb{R}^n$ is the decentralized control input to each system, defined as the consensus control law 
\begin{equation} \label{ui}
u_i := -\sigma \Big[ l_{i1}(x_i-x_1) + \ldots + l_{iN} (x_i-x_N) \Big],
\end{equation}  
where $l_{ij}$ are different non-negative real numbers denoting the individual interconnection weights and the scalar parameter 
$\sigma>0 $ is  the common coupling strength. 

 The control law \eqref{ui} is reminiscent of that commonly used in the literature on consensus control,  in which,  the coupling strength $\sigma = 1$.  This is specifically the case for networks of linear systems, in which case {\it complexity} hardly appears and the focus turns towards relaxing the various conditions pertaining to the nature of the interconnections. These may be linear \cite{WEIBOOK}; nonlinear \cite{arcak_TAC2007_passive-networks_4287131,isidori-marconi2014_6819823}; time-varying \cite{MoreauCDC04,chowdhury2016persistence}; switching \cite{olfati2004consensus,Morarescu-switching_9105084} or even state dependent \cite{al:SCLCHANNELS,casadei2019_8713390}. Other topology aspects,  such as whether the graph is directed \cite{martin2016time} or the interconnections are signed \cite{altafini2012_6329411},  may also alter consensus.  
 For networks of heterogeneous nonlinear systems,  however, the coupling strength plays a central role.  Different kinds of emergent behavior may arise depending on whether $\sigma$ is ``weak'' \cite{pogromsky2011_6160437,PANTELEY_SCHOLL} or ``strong'' \cite{DYNCON-TAC16,shim_AUT2020}.

Our main interest in this paper is to assess the behavior of the corresponding closed-loop system, specifically, in the case that the coupling gain $\sigma$ is larger than a certain threshold,  but we restrict our analysis to networks with an underlying static directed graph.  Akin to \cite{HDR-LENA} and \cite{DYNCON-TAC16},  we analyze the closed-loop networked system via a change of coordinates---introduced in \cite{CP3-CDC16-Mohamed}---that exhibits an intrinsic dichotomous structure composed of two dynamics defined in orthogonal spaces. On one hand, one has a 
{\it reduced-order dynamics} with state $x_m\in \mathbb{R}^n$ (miscalled {\it emergent dynamics} in \cite{DYNCON-TAC16}) and, on the other,  the dynamics of synchronization errors, denoted by $e_i := x_i - x_m$.  This characterization of the networked system is driven by the objective of characterizing synchronization phenomena that may appear (or not) in view of the systems' interconnections. 

In \cite{DYNCON-TAC16},  we say that the systems reach {\it dynamic consensus} if, for all $i\leq N$,  the synchronization errors $e_i$ converge to zero asymptotically.  The dynamic consensus paradigm generalizes the more common {\it equilibrium} consensus, in which case,  the reduced-order dynamics is null,  {\it i.e.} $\dot x_m = 0$,  because the collective behavior is {\it static} and the state of the reduced-order dynamics satisfies $x_m(t) \equiv x_m(0)$,  where $x_m(0)$ is a weighted average of the nodes' states' initial values. In the case of a network of oscillators, the reduced-order dynamics may admit an asymptotically stable equilibrium (an example is provided in \cite{DYNCON-TAC16}) or an asymptotically stable attractor \cite{JP2-STUART-LANDAU}.  However, in general, asymptotic dynamic consensus is unreachable due to the heterogeneity \cite{shim_AUT2020}.  An exception is that of systems that admit an internal model \cite{wieland-sep-AUT-2011,depersis2014_TCNS_internal-model,kim2011_05605658}. Otherwise, for general nonlinear heterogeneous systems, dynamic consensus may be guaranteed only in a \mbox{\it practical} sense \cite{DYNCON-TAC16}.

In this paper,  we establish two statements, each addressing one of two possible cases: one in which the (slow) reduced-order dynamics  admits a globally asymptotically stable equilibrium and another one,  in which,  it admits a periodic solution.   
In the first case,  we give sufficient conditions,  under which,  the origin for the networked system is globally asymptotically stable.  
In the second case,  we prove that the synchronization errors converge to a unique attractive periodic orbit,  so the systems synchronize in frequency.  Moreover,  for ``large'' values of the coupling strength $\sigma$,  this orbit is ``close''  to that generated by periodic solutions of the reduced dynamics. 
Thus,  the emergent dynamics approaches that of the  reduced-order system,  as the coupling gain grows. 

The analysis  is based on the recognized premise that in  
self-organized complex systems,  emergence is {\it multi}-level,  as  occurring in multiple timescales \cite{heylighen1989self,haken77synergetics}.  Here, we only consider two, one that is {\it slow} and pertains to the reduced-order system and another that is {\it fast} and pertains to the synchronization errors. We rely on singular perturbation theory,  on statements found in \cite{anosov1960limit,kokotovic1976singular,kokotovic1999singular,KHALIL96}, but also on original refinements of some statements from \cite{anosov1960limit} for systems admitting disconnected sets composed of equilibria and periodic orbits. 

The model-reduction-and-multi-time-scale perspective is certainly not new,  neither in systems theory \cite{chow1982} nor in other disciplines.  In the seminal work \cite{chow_kokotovic1985}, which follows up on \cite{chow1982}, the authors consider a modular network composed of sparsely connected clusters of densely interconnected dynamical systems modeled by simple integrators---the paradigm is motivated by that of large electrical networks.  Using classical singular-perturbation theory \cite{anosov1960limit,kokotovic1976singular},  it is showed that such networks achieve synchronization at two levels,  within and among the clusters. The analysis is based on relating the network's sparsity to a singular-perturbation parameter.  These concepts have been revisited in many succeeding works,  such as \cite{biyik2008area} and \cite{martin2016time}.  In the former, for networks of simple integrators through sector nonlinearities, and in the latter, for linear homogeneous systems interconnected through time-varying persistently-exciting gains {\it a la} \cite{MoreauCDC04}.  On the other hand, in   \cite{tognetti2021synchronization,rejeb2018control},  networks of linear homogeneous singularly-perturbed systems are considered. Thus, in all of the above,  the setting is fundamentally different from the one adopted here.

In \cite{CP3-CDC16-Mohamed}, for a particular case-study of networked Andronov-Hopf oscillators,  we use a coordinate transformation to exhibit the presence of the two-timescale emergent dynamics and singular-perturbation theory to analyze the collective behavior under the premise that the reduced-order system admits an asymptotically stable orbit.  Based on the coordinate transformation introduced in \cite{CP3-CDC16-Mohamed},  singular-perturbation theory is used in \cite{shim_AUT2020} on a wider class of nonlinear systems with rank-deficient coupling to establish synchronization in the practical sense.  Thus, there are several articles in the literature that explicitly use reduction and singular-perturbation theory,  even in a 
multi-agent context.  Yet,  we are not aware of any such work whose scope covers generic nonlinear heterogeneous systems interconnected over directed graphs and characterize the collective behavior with higher precision.  Conceptually,  in phase with the emergentism posit, we exhibit the emergence of a {\it complex} (in the sense of \cite{Arecchi1997}) dynamic behavior,  as a result of the systems' interactions. At the same time, we give a more precise characterization (well beyond practical asymptotic stability of the synchronization manifold) of the collective behavior of networked systems based on that of a reduced-order model. 

The rest of the paper is organized as follows.  In Section \ref{sec:model},  we exhibit the network's reduced-order and synchronization dynamics,  under an invertible  coordinate transformation.  In Sections \ref{sec:case1} and \ref{sec:case2},  we present our main results for the two cases described above, respectively. In Section \ref{sec:casestudy}, we revisit the case-study of Stuart-Landau oscillators,  and we conclude with some remarks in Section \ref{sec:concl}. The paper is completed with technical appendices. \color{black} 
 
 \textbf{Notation.}   
Given a nonempty set $K \subset \mathbb{R}^{n}$,  $|x|_K := \inf_{y \in K} |x-y|$, where  $|s|$ denotes the Euclidean norm of $s$, defines the distance between $x$ and the set $K$.   For a nonempty set $O \subset \mathbb{R}^{n}$, $K \backslash O$ denotes the subset of elements of $K$ that are not in $O$.    For a matrix $A \in \mathbb{R}^{n \times n}$,  $A^{-1}$ denotes its inverse, $A^\top$ denotes its transpose, and $|A|$ denotes its norm.  For a matrix $\Gamma \in \mathbb{R}^{n \times n}$,  $\lambda_{min}(\Gamma)$ and $\lambda_{max}(\Gamma)$ denote the smallest and the largest eigenvalues of $\Gamma$,  respectively.   By  $1_N  \in \mathbb{R}^{N}$,  we denote the vector whose entries are equal to $1$.   For a sequence  $\{ A_i \}^{N}_{i =1} \subset \Pi^{N}_{i=1} \mathbb{R}^{n_i \times n_i}$,  $\underset{i \in \{1,2,...,N \}}{\text{blkdiag}} \hspace{-0.2cm} \{A_i\}$ is the block-diagonal matrix whose $i$-th diagonal block corresponds to the matrix $A_i$.  By $\otimes$,  we denote the Kronecker product.  For a complex number $\lambda \in \mathcal{C}$, $\Re{e}(\lambda)$ denotes the real part of $\lambda$ and $\Im(\lambda)$ denotes the imaginary part of $\lambda$.  

\section{On strongly-coupled connected networks } \label{sec:model}

\subsection{The model and standing assumptions}

Consider a group of $N$ nonlinear systems as in \eqref{dxifi} driven by the distributed control inputs in \eqref{ui}, where each $l_{ij}\geq 0$ is constant but not necessarily equal to  $\l_{ji}$.   
In particular,  when $\l_{ij}$ is strictly positive,  then there exists an interconnection from the $j$th node to the $i$th node, but $\l_{ji}$ may be null, in which case, the interconnection is unidirectional.  More precisely, we pose the following hypothesis.

\begin{assumption}[connected di-graph] \label{ass3} 
The network's graph is connected. 
\end{assumption}

\begin{remark}\label{rmk1}
We stress that under Assumption \ref{ass3}, any kind of directed graph containing a rooted spanning-tree (with or without cycles) is considered.  Moreover, under Assumption \ref{ass3} the Laplacian $L$ has exactly one eigenvalue (say, $\lambda_1$) that equals  zero, while the others have positive real part, {\it i.e.}, 
$ 0 = \lambda_1 < \Re{e} \left\{ \lambda_2 \right\} \leq \ldots \leq \Re{e}\left\{ \lambda_n \right\}$.  Furthermore, the right eigenvector corresponding to the simple eigenvalue $\lambda_1 =0$, is $v_r= 1_N  \in \mathbb{R}^N$, while the left eigenvector, denoted $v_l$, contains only non-negative elements \cite{olfati2004consensus} and satisfies ${1}^\top_N v_l=1$.  
\end{remark}

In addition,  for each unit,  we impose the following regularity and structural hypotheses: 

\begin{assumption}[equilibrium] \label{ass1} 
The functions  $f_i$ are continuously differentiable and, without loss of generality,  we assume that $f_i(0) = 0$ for all $i \in \{1,2,\ldots,N\}$.
\end{assumption}

\begin{assumption}[Semi passive units] \label{ass2}
Each agent is input-to-state strictly semi-passive; 
namely,   for each $i\in\{1,2,\ldots,N\}$, there exists a continuously differentiable and radially unbounded storage function $V_i: \mathbb{R}^n \rightarrow \mathbb{R}_{\geq 0}$,  a positive constant $\rho$,  a continuous function $H_i : \mathbb{R}^n \rightarrow \mathbb{R}$,  and a continuous function $\psi_i : \mathbb{R}^n \rightarrow \mathbb{R}_{\geq 0}$ such that
\begin{align*}
\dot{V}_i(x_i) \leq x^{T}_i u_i - H_i(x_i) 
\end{align*}  
and $ H_i(x) \geq \psi_i(\left| x \right|) $ for all 
$ \left| x \right| \geq \rho $.
\end{assumption}

Assumption \ref{ass2} is useful to assess the boundedness of solutions for system \eqref{dxifi} in closed loop with \eqref{ui} for a sufficiently large coupling strength $\sigma$.   More precisely,  we have the following result.   

\begin{lemma} [Global ultimate boundedness \cite{DYNCON-TAC16}]  \label{lem0}  
Consider the systems in \eqref{dxifi} in closed loop with the control inputs in \eqref{ui} and let Assumptions \ref{ass3}--\ref{ass2} be satisfied.  Then,  the closed-loop system is (uniformly in $\sigma$) ultimately bounded; namely,   there exist $\sigma^\ast>0$ and  $r>0$ such that, for any $R \geq 0$, there exists $\tau_{R} \geq 0$  such that, for each solution $x(t)$ starting from $x_o \in \mathbb{R}^{nN}$, we have  
\begin{align*} 
\left|{x}_o\right| \leq R \Longrightarrow \left|{ x}(t)\right| \leq r  \qquad \forall t \geq \tau_{R},\quad \forall \sigma \geq  \sigma^\ast,
\end{align*} 
where ${x} \in \mathbb{R}^{nN}$ denotes the network's state,  i.e., $x = [x_1,\ldots, x_N]^\top$.
\end{lemma}

Under the assumptions listed above, we investigate the problem of assessing the behavior of the networked closed-loop system \eqref{dxifi}-\eqref{ui}. To this end,  as it is customary, let us collect the individual interconnection coefficients $l_{ij}$ into the Laplacian matrix ${ L}:=[\ell_{ij}] \in{\mathbb{R}}^{N\times N}$, where 
\begin{equation*}
{\ell _{ij}} = \left\{ {\begin{array}{*{20}{c}}
{\sum\limits_{k \in  \mathcal N_i} {{a_{ik}}} }&{i = j}\\
{ - {a_{ij}}}&{i \ne j}.
\end{array}} \right.
\end{equation*}
Then, replacing  \eqref{ui} in  \eqref{dxifi},  we see that the overall network dynamics takes the form 
\begin{equation}
\label{ntwksyst}
\dot{x} =  F({x}) - \sigma [L \otimes I_n] {x}, 
\end{equation}
where the function $F: \mathbb{R}^{nN} \to \mathbb{R}^{nN}$ is given by 
\begin{eqnarray*}
F({x})&:=&\left[ f_1(x_1)\ \cdots\ f_N(x_N) \right]^\top.
\end{eqnarray*}

As in \cite{DYNCON-TAC16} and \cite{HDR-LENA}, to analyze the  behavior of the network system \eqref{ntwksyst}, we acknowledge its dichotomous nature. In these references as well as in many others---see, e.g., \cite{montenbruck2015practical,hill2015,delellis2011quad}, synchronization is defined as the property of the trajectories of each individual system following the trajectories of an ``averaged'' unit with state 
\begin{equation}
  \label{221} x_m := [ v_l^\top \otimes I_n ] x.
\end{equation}
The quotes in ``averaged'' are superfluous in the case of undirected networks, in which case $v_l = 1_N$, so $x_m = \frac{1}{N} \sum_{i=1}^N x_i$, but for directed connected networks the state $x_m$ is more generically defined as a weighted average of the respective systems' states since ${v_l}_i \geq 0$ for all $i \in \left\{ 1,2,...,N \right\}$ and $ v_l^\top 1_N = 1$.

In either case,   a sensible way to define the synchronization errors $e$ is as the difference between the units' states and $x_m$,  that is,
\begin{equation} \label{301} 
e := {x} - [ 1_N \otimes I_n ] x_m .
\end{equation}
Thus,  in \cite{DYNCON-TAC16} and \cite{HDR-LENA} the collective behavior of network systems is studied in function of the dynamics of the ``averaged'' unit $x_m$ and that of the synchronization errors $e$. 

In the next section, we introduce another change of coordinates to rewrite  system \eqref{ntwksyst} in an equivalent form that exhibits two motions;  one that is generated by the averaged dynamics and another by a {\it projection} of the synchronization errors $e$ on a certain subspace.   This coordinate transformation is not a simple artifice for analysis, it exhibits two time-scales that are inherent to networked systems satisfying Assumption \ref{ass3} and subject to a sufficiently large coupling $\sigma$.  

\subsection{Intrinsic two-time-scales decomposition}

After Assumption \ref{ass3} and Remark \ref{rmk1}, because $\lambda_1=0$ has multiplicity one, \color{black} the Laplacian admits the following Jordan-block decomposition of over $\mathbb{R}^{N \times N}$: 
\begin{align} \label{A3}
L = U \begin{bmatrix} 0 & 0 \\ 0 & \Lambda \end{bmatrix} U^{-1},
\end{align}
where $ \Lambda \in \mathbb{R}^{(N-1) \times (N-1)}$ is composed by the Jordan blocks corresponding to the  $N-1$ non-zero  eigenvalues. 

\begin{remark}
Note that even though a Jordan decomposition does not necessarily exist with a real matrix $U$,  it is always possible to use the spectral-invariant-subspace decomposition as in \Lemspecres\ to generate a real matrix $U$.
\end{remark}

The convertible matrix $U$ is constituted, column-wise, of the right eigenvector of the Laplacian, $1_N$,  and a left-invertible  matrix  $V\in \mathbb{R}^{N \times (N-1)}$, which consists of the eigenvectors corresponding to the nonzero eigenvalues of $L$. That is, 
\begin{align} \label{eqU}
U = [1_N~~V], \qquad U^{-1} = \begin{bmatrix} v_{l}^\top \\ V^\dag  \end{bmatrix},
\end{align}
where  $V^{\dag} \in \mathbb{R}^{(N-1) \times N}$, and 
\begin{align} \label{eqprod}
v_l^\top V = 0,~~~ V^\dag V = I_{N-1}.
\end{align}
So, using \eqref{eqU} and \eqref{eqprod}, we also have the useful identity
\begin{align*}
	V V^\dag = I_N - 1_N v_l^\top.
\end{align*}

Now, using $U^{-1}$,  we define the new coordinates 
\begin{eqnarray}
\label{barx}
\bar x := [ U^{-1} \otimes I_n ] x
\end{eqnarray}
and their inverse transformation
\begin{equation}
\label{invbarx}
  x := [ U \otimes I_n ] \bar x.
\end{equation}
The interest of the coordinate $\bar x$ is that it consists in the familiar ``averaged'' states $x_m$ and a  projection of the synchronization errors $e$ defined in \eqref{301} onto the subspace that is generated by $V^\dagger$, which is orthogonal to the right eigenvector  $1_N$. To better see this, note that such projection yields 
\begin{align*}
  [ V^\dag \otimes I_n ]  e & =  [ V^\dag \otimes I_n ] \big[ x - [ 1_N \otimes I_n ] x_m \big] \\
\end{align*}
In the sequel, we refer to the left-hand side of the latter equation as the projected synchronization errors, 
\begin{equation}
  \label{366} e_v :=  [ V^\dag \otimes I_n ] x. 
\end{equation}
Hence,  in view of \eqref{221}, \eqref{eqU}, \eqref{barx}, and \eqref{366}, we have
\begin{eqnarray}
\label{282} 
\bar x = \left[ 
            \begin{array}{c} 
              x_m \\ e_v 
            \end{array} 
          \right] = 
 \left[ \begin{array}{c}  
     [ v_l^\top \otimes I_n ] x \\[2pt]
     [ V^\dag \otimes I_n ] x 
     \end{array} 
 \right].
\end{eqnarray}

In the new coordinates, the network system \eqref{ntwksyst} is equivalently written as 
\begin{align*}
\dot{\bar{x}}= [ U^{-1} \otimes I_n ] \big[ F(x) - \sigma [ L \otimes I_n ] x \big],
\end{align*}
which consists in two interconnected dynamics, that of the ``averaged'' state  $x_m$ and that of the projected synchronization errors $e_v$. Therefore, the behavior of the trajectories of \eqref{ntwksyst} may be assessed via the behavior of the latter dynamics.  To this end,  we use  $\bar{x} = [x_m^\top ~~ e_v^\top]^\top$ and $U = [1_N~~V]$ in \eqref{invbarx} to write
\begin{equation}  \label{293} 
   x =  [1_N \otimes I_n ] x_m  + [V \otimes I_n] e_v.
\end{equation}

So, differentiating on both sides of \eqref{221},  using \eqref{ntwksyst}, \eqref{293}, and the fact that $v_l^\top L = 0$, we obtain
\begin{equation}  \label{A8}  
  \dot x_m = {F}_m(x_m) + G_m(x_m,e_v),
\end{equation}
where $F_m(x_m) := [ v_l^\top \otimes I_n ] F \big(\, [1_N \otimes I_n ] x_m \big) $ and 
\begin{align*}
G_m(x_m,e_v)  & :=  [ v_l^\top \otimes I_n ] 
   \big[ F \big(\,   [ 1_N \otimes I_n ] x_m  + [V \otimes I_n] e_v \big)   
     \\ & \hspace{10.8em}  - F \big(\, [1_N \otimes I_n ] x_m \big)\, \big].
\end{align*}
Note that $F_m(x_m)$ effectively corresponds to an ``averaged'' dynamics of the systems in \eqref{dxifi},  
$$
{F}_m(x_m) = \sum_{i=1}^N v_{li}f_i(x_m),
$$
$G_m(x_m,0) = 0$ and, under Assumption \ref{ass1}, all these functions are smooth and there exists a continuous function $h: \mathbb{R}^{nN}\to \mathbb{R}_{\geq 0}$ such that  
\begin{equation} \label{ineqGm}
\big| G_m(x_m,e_v) \big|  \leq  h(x_m,e_v)|e_v|  \qquad  \forall (x_m,e_v) \in \mathbb{R}^{nN}.
\end{equation}

On the other hand, by differentiating on both sides of $e_v = [ V^\dag \otimes I_n ] x$ and using \eqref{A3}, \eqref{eqU}, and \eqref{293}, we obtain the {\it synchronization-errors  dynamics} 
\begin{equation}  \label{A9} 
  \dot{e}_v  =  - \sigma [ \Lambda \otimes I_n ] e_v + G_e(x_m,e_v) ,
\end{equation}
where 
$$ G_e(x_m,e_v) := [ V^\dag \otimes I_n ] F\big( \, [ V \otimes I_n ] e_v + [ 1_N \otimes I_n ] x_m \big).  $$

The complete collective behavior of the networked control system \eqref{ntwksyst}, up to the globally invertible coordinate transformation in \eqref{barx}, may be assessed by analyzing that of the interconnected systems \eqref{A8} and \eqref{A9}. We see that the systems in \eqref{dxifi} under the action of the control laws in \eqref{ui} {\it synchronize} if the errors $e_v$ tend asymptotically to zero.   However, the characterization of the networked systems' behavior would be incomplete unless one can ascertain what the individual systems do {\it when} they synchronize.  
Indeed, {\it a priori}, not even boundedness of solutions is guaranteed 
(whence Assumption \ref{ass2}). 
To assess any kind of stable behavior, we analyze the network system \eqref{ntwksyst} {\it on} the synchronization subspace corresponding to $e_v = 0$.  On such a subspace, we have the 
{\it reduced-order dynamics}
\begin{equation} 
\label{redsyst}
\dot{x}_m  = F_m(x_m).  
\end{equation}

So,  it is clear that the motion of the synchronized systems is fully determined by that of the reduced-order dynamics \eqref{redsyst}. In this regard, it is important to underline that \eqref{redsyst}, as well as the ``averaged'' dynamics \eqref{A8} are independent of the coupling gain $\sigma$. This dynamics is {\it inherent} to  the network and appears simply as a consequence of the graph's connectivity imposed by Assumption \ref{ass3}. The synchronization dynamics \eqref{A9}, on the other hand, clearly depends on the coupling strength $\sigma$. In this paper we are interested in investigating the synchronization behavior for `large' values of the coupling strength. More precision about the meaning of `large' is given farther below.

We consider two scenarii of major interest. The first pertains to the case in which the reduced-order dynamics \eqref{redsyst} admits the origin as a globally asymptotically stable equilibrium point. This case covers, in particular, the classical problem of consensus for heterogeneous nonlinear systems interconnected over generic connected graphs, since in this case $\dot x_m \equiv 0$. Our main statement in this case (Theorem \ref{thm:gas}) is that not only the networked system achieves dynamic consensus but the origin $\{ x = 0\}$ is GAS for \eqref{ntwksyst}.  The second scenario pertains to the case in which the reduced-order dynamics  admits an unstable equilibrium and a stable periodic orbit.  Our main statement in this case (Theorem \ref{thm2}) establishes sufficient conditions for almost global asymptotic stability.  For instance, for a network of periodic heterogeneous nonlinear oscillators, it is possible to assess the conditions under which they synchronize {\it and} to characterize the resulting collective oscillatory behavior. 

The analysis of \eqref{ntwksyst} is carried out using singular-perturbations theory.  For this, we write the network system \eqref{ntwksyst} in the familiar singular-perturbation form \cite{kokotovic1999singular,KHALIL96} 
\begin{subequations}\label{singpertsyst}
  \begin{eqnarray}
\label{A10} \dot{x}_m &\hspace{-4pt} = &\hspace{-4pt} F_m(x_m) + G_m(x_m,e_v)   \\ 
\label{A11} \varepsilon \dot e_v &\hspace{-4pt} = &\hspace{-4pt} - \left( \Lambda \otimes I_n \right) e_v 
                   +  \varepsilon  G_e(x_m,e_v), \quad \varepsilon := 1/\sigma, \qquad \
\end{eqnarray}
\end{subequations}
in which we recognize two time scales, ``slow'' and ``fast'', corresponding, respectively, to the dynamics of the averaged-unit states $x_m$ and the projected synchronization errors $e_v$.  Now,  in accordance with singular-perturbation theory; see  \cite[p.\,358]{ KHALIL96},  the behavior of \eqref{singpertsyst} is ineluctably determined by that of the {\it slow} dynamics, obtained by setting $\varepsilon =0$,  which clearly corresponds to the reduced-order model \eqref{redsyst}.  Thus, the rest of the paper is devoted to the analysis of \eqref{singpertsyst} in the two cases evoked above. 

Even though the analysis relies on the study of the system in singularly-perturbed form, \eqref{singpertsyst}, we remark that our two main statements are formulated for system \eqref{ntwksyst}, which remains the main subject study in this paper.  Therefore, we finish this section by re-expressing the properties of \eqref{ntwksyst} in Assumptions \ref{ass3}--\ref{ass2}  in terms of \eqref{singpertsyst}, in the form of the following, rather evident, statement that is extensively used in the sequel.
\begin{lemma}   \label{fact1}
Consider system \eqref{ntwksyst} such that  Assumptions \ref{ass3}--\ref{ass2} hold.  Then,  the resulting system \eqref{singpertsyst},  with states defined in \eqref{282},  enjoys the following properties: 
\begin{itemize}
\item[(i)] the functions $F_m$, $G_m$, and $G_e$ are  continuously differentiable;
\item[(ii)] the origin $\{(x_m,e_v) = (0,0)\}$ as an isolated equilibrium point;
\item[(iii)] the solutions to \eqref{singpertsyst} are globally uniformly (in $\sigma$) ultimately bounded; 
\item[(iv)] the matrix $\Lambda$ is Hurwitz.\\[-25pt] 
\end{itemize}
\end{lemma}

\section{Case I: Global asymptotic stability}\label{sec:case1}

Consider the network system \eqref{ntwksyst} in its equivalent singular-perturbation representation \eqref{singpertsyst}.  In the case that the reduced-order dynamics  \eqref{redsyst} admits the origin as a globally asymptotically stable  equilibrium, under Assumptions \ref{ass3}--\ref{ass2}, global asymptotic stability  of the origin for  \eqref{singpertsyst}  follows  under a sufficiently small perturbation 
$\varepsilon>0$.   According to \cite{teel1999semi},  we recall the  practical-stability notions of a closed subset $\mathcal{A} \subset \mathbb{R}^n$ for a general nonlinear system of the form
\begin{align} \label{eqPstab}
\dot{x} = f(x, \varepsilon) \qquad x\in \mathbb{R}^n, \quad \varepsilon\in [0,1].  
\end{align}
\begin{itemize}
\item The set $\mathcal{A}$ is globally practically attractive if,  for each $\beta > 0$, there exists  $\varepsilon^{\star} > 0$ such that, for each $\varepsilon \leq \varepsilon^\star$,  every solution $x(t)$ to \eqref{eqPstab},  there exists $T>0$ such that  $||x(t_o + T)||  \leq \beta$.
\item The set $\mathcal{A}$ is  globally practically stable if there exists $\kappa \in \mathcal{K}$ such that,  for each $\beta > 0$,  there exists  $\varepsilon^{\star} > 0$ such that, for each $\varepsilon \geq \varepsilon^\star$,  we have 
$$ |x(t)| \leq \kappa \left( |x(t_o)|  \right) + \beta  \quad \forall t \geq t_o.  $$
\item The set $\mathcal{A}$ is globally asymptotically practically stable if it is globally practically attractive and globally practically stable.   
\end{itemize}

Our first statement is the following. 

\begin{theorem}[GAS]\label{thm:gas}  
Consider system \eqref{ntwksyst}  under Assumptions \ref{ass3}--\ref{ass2}.  In addition, assume that for  system \eqref{redsyst}, the origin $\{ x_m = 0 \}$  is globally asymptotically stable.   Then, 
\begin{itemize}
  \item[(i)] the origin for \eqref{ntwksyst} is globally asymptotically practically stable.
\end{itemize}
Furthermore, assume in addition that there exists a continuously differentiable Lyapunov function $ V_m : \mathbb{R}^n \rightarrow \mathbb{R}_{\geq 0}$ and a class $\mathcal K_\infty$ function $\alpha$ such that 
\begin{align} \label{eqder}
\frac{\partial V_m(x_m)}{\partial x_m^\top} F_m(x_m)   \leq - \alpha\left( \left| x_m \right| \right)^2
\end{align}
and there exists $c_r>0$ such that, for all $\bar x \in B_r$---see \mbox{\eqref{282}},
\begin{align} \label{eqder1}
\displaystyle\max_{\bar x \in B_r} \left\{  \big| G_e(x_m,e_v) - G_e(0,e_v) \big|, \left| \frac{\partial V(x_m)}{\partial x_m} \right|  \right\} \leq  c_r \alpha \left( \left| x_m \right| \right),
\end{align}
where $r > 0$ is the global ultimate bound established in Lemma \ref{lem0}. Then, 
\begin{itemize}
\item[(ii)] there exists $\sigma^\star>0$ such that, for all $ \sigma \geq \sigma^\star $, the origin for  \eqref{ntwksyst}  is globally asymptotically stable.  \\[-25pt] 
\end{itemize}
\end{theorem}

\begin{proof}
The statement in Item (i) follows by establishing global asymptotic practical stability of the origin $\{ \bar x = 0 \}$ for  \eqref{singpertsyst}. To this end, we first remark that, after Assumptions \ref{ass1}-\ref{ass2} and  Lemma \ref{fact1},   system \eqref{singpertsyst} is globally uniformly (in $\sigma$) ultimately bounded.  That is,  there exist $\sigma^\ast>0$ and  $r>0$ such that, for any $R \geq 0$, there exists $\tau_{R} \geq 0$  such that  
\begin{align*} 
\left|\bar x_o \right| \leq R \Rightarrow \left| \bar x(t) \right| \leq &\, r  \quad \forall t \geq \tau_{R}, ~  \forall \sigma \geq  \sigma^\ast.
\end{align*} 
Hence, it suffices to show (i) only inside the  compact set   
$B_r$.   To do so,  we show that,  for every positive constant $r_f < r$,  there exists $\sigma_f>0$ such that, for each $\sigma \geq \sigma_f$,   the compact set $B_{r_f}$ is globally asymptotically stable for \eqref{singpertsyst} on $B_r$ \cite{teel1999semi}.   We first rewrite the perturbed system \eqref{singpertsyst} in the following form
\begin{align}
\dot{x}_m = & {F}_m(x_m) +  {G}_m(x_m,e_v)   \nonumber  \\ 
\varepsilon \dot{e}_v = & - \left( \Lambda \otimes I_n \right) e_v + \varepsilon G_e(0,e_v) 
\nonumber \\ & + \varepsilon \left[ G_e(x_m,e_v) -G_e(0,e_v) \right].   \label{A11biz}
\end{align}
Next,  using  Assumption \ref{ass1}, we conclude the existence of a constant $d_r > 0$ such that, for each $\bar x \in B_r$, 
\begin{align} 
 \max_{\bar x \in B_r}  \left\{ \left| {G}_m(x_m,e_v) \right|, \left| G_e(0,e_v) \right| \right\}  & \leq d_r \left| e_v \right|,  \label{eq:born1}
\\
\left|  G_e(x_m,e_v) - G_e(0,e_v) \right|  & \leq  d_r \left| x_m \right| \qquad  \forall \bar x \in B_r.
   \label{eq:born2}
\end{align}

Furthermore,  after Item (iv) in Lemma \ref{fact1},  the linear system 
$$ \dot{e}_v = - \left( \Lambda \otimes I_n \right) e_v $$ 
is exponentially stable.  As a result,  there exists $P \in \mathbb{R}^{(N-1) \times (N-1)}$  symmetric and positive definite such that 
$$  P \Lambda  + \Lambda^\top P \leq - I_N.  $$

Finally,  we use the converse Lyapunov theorem in  
\cite[Theorem 3.14]{KHALIL96} to conclude that,  since the origin for \eqref{redsyst} is globally asymptotically stable,  there  exists a function $V_m : \mathbb{R}^n \rightarrow \mathbb{R}_{\geq 0}$ and positive constants $\alpha_{r r_f}$ and  $\beta_{r r_f}$  such that,  for each $\bar x \in B_r \backslash B_{r_f}$,  we have 
\begin{eqnarray}
   \label{481} \left|\frac{\partial V_m(x_m) }{\partial x_m} \right| &\  \leq\ & \beta_{r r_f}\left|x_m\right|,
\\
\frac{\partial V_m(x_m) }{\partial x_m^\top} F_m(x_m)  &\ \leq\ & - \alpha_{r r_f} \left|x_m\right|^2.
   \label{484}
\end{eqnarray}
Therefore, under \eqref{eq:born1}--\eqref{484},  the total derivative of the  Lyapunov function candidate   
\begin{align} \label{eqVVV}
V(\bar x) := V_m(x_m) + V_e(e_v), 
\end{align}
where  $V_e(e_v) := e_v^\top (P \otimes I_n) e_v$, satisfies 
\begin{align*}
\dot V(\bar x)  \leq & - \alpha_{r r_f} |x_m|^2 
+ d_r \left[ \lambda_{max}(P) + \beta_{r r_f} \right]  \left| x_m \right| \left| e_v \right| 
\\ & - \left[ \sigma - d_r \lambda_{max}(P) \right] \left| e_v \right|^2 \qquad  \forall\, \bar x \in B_r \backslash B_{r_f} .
\end{align*}
Thus,   for $\sigma_f := d_r^2 \left( \lambda_{max}(P) + \beta_{r r_f}\right)^2 + 2 d_r \lambda_{max}(P) $, we conclude that,  for each $\sigma \geq \sigma_f$,  we have 
\begin{align*}
\dot{V}(e_v, x_m) \leq  - \frac{1}{2} \alpha_{r r_f} |x_m|^2 - \frac{1}{2} \sigma \left|e_v\right|^2 \qquad  \forall \, \bar x \in B_r \backslash B_{r_f}.
\end{align*}

The statement  in Item (ii) \color{black} follows by establishing global asymptotic stability of the origin $\{ \bar x = 0 \}$ for  \eqref{singpertsyst}.   However,  since \eqref{singpertsyst} is globally uniformly (in $\sigma$) ultimately bounded,  it is enough to prove that the origin is asymptotically stable on for all solutions contained in $B_r$.    So in the rest of the proof,  we restrict the analysis on $B_r$.   Let us reconsider the Lyapunov function $V$ introduced in \eqref{eqVVV}.  Note that the time derivative of $V$ along the solutions to \eqref{singpertsyst} satisfies 
\begin{equation}
\label{dotV}
\begin{aligned}
\dot V(\bar x) &= 
- \alpha(|x_m|)^2  +\frac{\partial V_m}{\partial x_m^\top} G_m(x_m,e_v)
  - \sigma \left|e_v\right|^2 
 \\ &   + 2 e_v^\top P \left[ G_e(0,e_v) + \left( G_e(x_m,e_v) - G_e(0,e_v) \right)  \right].
\end{aligned}
\end{equation}

Now,  after \eqref{eq:born1} and \eqref{eqder1},  we conclude that
\begin{equation}
\label{eqbounddd}
\begin{aligned} 
\frac{\partial V(\bar x) }{\partial x_m} G_m (x_m,e_v)  &  \leq c_r d_r \alpha \left( \left| x_m \right| \right)  \left| e_v \right|  \qquad  \forall\, \bar x \in B_r, 
\end{aligned}
\end{equation}
and, consequently, from \eqref{dotV}, \eqref{eq:born1}, and \eqref{eqbounddd},  it follows that
\begin{align*}
\dot V(\bar x) \leq & - \alpha (|x_m|)^2 + c_r [2 \lambda_{max}(P) + d_r] \alpha( | x_m | )| e_v | \\
&  -  [\sigma - 2d_r \lambda_{max}(P) ]| e_v |^2 \qquad  \forall \bar x \in B_r.
\end{align*} 
Therefore,  for  $\sigma^{\star} := c_r^2 [ \lambda_{max}(P) + d_r ]^2 + 4 d_r \lambda_{max}(P)$,  we conclude that, for each $\sigma \geq \sigma^\star$,
\[   \dot V(\bar x) \leq  - \frac{1}{2} \alpha (|x_m|)^2 - \frac{1}{2} \sigma \left|e_v\right|^2 \qquad  \forall\, \bar x \in B_r.  \]
\end{proof}

\begin{remark}
 The first statement in Theorem \ref{thm:gas} can be deduced using \cite[Corollaries 1 and 3]{DYNCON-TAC16},  a detailed proof is included in this paper  for the sake of self containedness.  
 \end{remark}
 
 \begin{remark}
 The regularity conditions in \eqref{eqder}-\eqref{eqder1} are required to ensure negativity of the time derivative of the Lyapunov function $V$ along the solutions to \eqref{singpertsyst}.    Even though they may appear conservative, since they are certainly not necessary ---see \cite[Exercise 9.24]{KHALIL96},  it is important to stress that the origin is not necessarily globally asymptotically stable if these conditions do not hold.  In \SectionAH\ we provide an example that illustrates this claim. Furthermore, the inequalities \eqref{eqder}-\eqref{eqder1} are required to hold only in a compact set containing the origin.
\end{remark}

\section{Case II: Periodic behavior} \label{sec:case2}  

In this section,  we present our second and main statement,  which pertains to the case where \eqref{redsyst} admits a periodic orbit that is attractive from almost all initial conditions. Under this condition,  Theorem \ref{thm2} below establishes that for $\sigma > 0$ sufficiently large,  the network system \eqref{ntwksyst} also admits a unique periodic orbit and it is globally attractive from almost all initial conditions.   In particular,  frequency synchronization is achieved and the synchronization errors can be made arbitrary small by choosing $\sigma$ sufficiently large.  It is important to stress that our main statement  establishes  a precise periodic behavior for the network system \eqref{ntwksyst} rather that just approaching the periodic solution to \eqref{redsyst}. 

For completeness and clarity,  we start by recalling some notions and tools related to the stability of periodic solutions to nonlinear systems.   

Consider the system
\begin{equation} \label{127bis}
\dot x = f(x) \qquad x \in \mathbb{R}^n,
\end{equation}
where $f  : \mathbb{R}^n \rightarrow \mathbb{R}^n$ is locally Lipschitz. 

\begin{definition}[Periodic solution and periodic orbit] 
A solution $t\mapsto\phi(t)$, or simply $\phi(t)$,  to \eqref{127bis} is said to be  $\alpha$-periodic if there exists  $\alpha>0$ (the period) such that,  for each $t\geq 0$, 
$$ \phi(t+\alpha)=\phi(t) \quad\mbox{\rm and}  \quad \phi(t+s)\neq \phi(t)   \quad \forall  s\in (0,\alpha).   $$

Moreover,  if the system  \eqref{127bis} admits a periodic solution 
$\phi$,  we say that it admits a (closed) periodic orbit $\gamma \subset \mathbb{R}^{n}$ generated by the image of  $\phi$.  
\end{definition}

Then,  according with Lyapunov theory, we may single out the following desired properties for periodic solutions. 

\begin{definition}[Orbital Stability] 
Let $\gamma$ be a periodic orbit for \eqref{127bis}. 
\begin{itemize}
\item The orbit $\gamma$  is  orbitally stable if,  for each $ \varepsilon > 0$,  there exist $\delta > 0$ and $T \geq 0$,  such that, for each initial condition $x_o$ satisfying  $|x_o|_\gamma \leq \delta$, the solution $\phi$ starting from $x_o$ satisfies  $ | \phi(t) |_\gamma \leq \varepsilon$ for all $t \geq T$.  

\item The orbit $\gamma$ is orbitally asymptotically stable,  if it is orbitally stable and attractive; {\em i.e.},  if there exists $R \in (0,+\infty]$ such that,  for each  $x_o$ satisfying  $|x_o|_\gamma \leq R$, the solution $\phi$ starting from $x_o$ satisfies   $\lim_{t \rightarrow \infty}  | \phi(t) |_\gamma  = 0$.

\item The orbit $\gamma$ is globally orbitally asymptotically stable if it is orbitally asymptotically stable with $R = + \infty$ and almost globally asymptotically stable if it is orbitally asymptotically stable for all $x_o \in \mathbb{R}^n\backslash \mathcal D$, where $\mathcal D \subset \mathbb{R}^n$ has a null Lebesgue measure. 
\end{itemize}
\end{definition}

Finally,   we recall some orbital stability criteria in terms of the so-called characteristics multipliers \cite[Section~III.7]{HALE_BOOK} which, for linear periodic systems, are the counterpart of eigenvalues for linear autonomous systems.  To see this,  we assume that $f$ is continuously differentiable and we consider the $\alpha$-periodic matrix $A(t) :=  \frac{\partial f}{\partial x^\top}(\phi(t))$,  
where $\phi(t)$ is the $\alpha$-periodic solution to \eqref{127bis} generating the orbit $\gamma$.   After Floquet theory---see {\it e.g.}, \cite{FLOQUET} and \cite{Perko}, there exist an $\alpha$-periodic non-singular matrix $P: [t_o,+\infty] \rightarrow \mathbb{R}^{n \times n}$ and a constant matrix $B \in \mathbb{R}^{n \times n}$ such that the transition matrix associated to the linear time-varying system 
\begin{align}  \label{charper}
\dot{x} = A(t) x 
\end{align}
is given by $	X(t) := P(t) e^{B t}$ and the non-singular change of coordinates $y := P(t)^- x$ transforms the linear system \eqref{charper} into $\dot{y} = B y$.

\begin{definition} [Characteristic multipliers]
The characteristic multipliers of the $\alpha$-periodic matrix $A(t)$ are the eigenvalues of the matrix $e^{B \alpha}$.  
\end{definition}

\begin{definition} [Non-singular periodic orbit]
The periodic orbit $\gamma$ generated by the periodic solution 
$\phi(t)$ is non-singular if the matrix $A(t) :=  \frac{\partial f}{\partial x^\top}(\phi(t))$ admits a simple characteristic multiplier equal to $1$.
\end{definition}

\begin{lemma}[Theorem 2.1, Section~VI.2.~\cite{HALE_BOOK}]  \label{lemfloq}
Consider  system \eqref{127bis} with $f$ continuously differentiable and let $\phi$ be a non-trivial $\alpha$-periodic solution generating the orbit $\gamma$.  Assume that the matrix $A(t) :=  \frac{\partial f}{\partial x^\top}(\phi(t))$ is non-singular and all the characteristic multipliers, except one, have modulus strictly less than $1$.  Then,  the resulting periodic orbit $\gamma$ is asymptotically orbitally stable.  
\end{lemma}

\subsection*{Sufficient conditions for orbital stability}

As mentioned above, generally speaking, the standing assumption in this section is that the  reduced-order dynamics  \eqref{redsyst} admits an  orbitally asymptotically stable periodic solution. However, we remark that some nonlinear systems defined on compact and convex sets and that admit a limit cycle, \color{black}also admit at least one equilibrium point \cite{BASENER2006841}.  This imposes particular richness to the network's collective behavior and considerable difficulty to analyze it since it translates into  studying the stability of a disconnected invariant set.  In that light, we pose the following, more precise hypothesis.  

\begin{assumption} \label{ass4}
 The reduced-order dynamics  \eqref{redsyst} admits a unique compact invariant subset  $\omega \subset \mathbb{R}^n$
 that is  globally attractive;  namely,  for each $x_{m o} \in \mathbb{R}^n$, the solution $x_m(t)$ starting from $x_{m o}$ satisfies 
\begin{align} \label{eqlimsup}
\limsup_{t\to+\infty} |x_m (t)|_\omega = 0
\end{align} 
 
Furthermore,  the set $\omega$ is compose  a non-singular periodic orbit 
$\gamma_o$,  of period $\alpha_o$,  that is orbitally asymptotically stable and 

 \begin{itemize}
\item the origin $\{x_m=0\}$,  when the latter is repulsive. 

\item the  homoclinic orbit $\gamma_1 := W^u_o(0) \cap W^s_o(0)$  when the origin is hyperbolic\footnote{\textit{i.e.} the eigenvalues of $A:=\frac{ \partial f_m(x_m)}{ \partial  x_m}  \big|_{x_m = 0}$ has $0<k<n$ eigenvalues with positive real part and $n-k$ with  negative real part. },  where $W^s_o(0)$ and $W^u_o(0)$ are,  respectively,  the global stable and  unstable manifolds of the origin.
\end{itemize}
\end{assumption} 

\begin{remark}
Note that the global attractivity property in \eqref{eqlimsup} plus the structure of the invariant set $\omega$ imply the existence of a Lyapunov function enjoying useful properties along the solutions to \eqref{redsyst}; see \lemAE.
\end{remark}

\begin{lemma}   \label{lemalgstab}
Under Assumption \ref{ass4} and  Item (i) in Lemma \ref{fact1},  the periodic orbit  $\gamma_o$ is almost globally orbitally asymptotically stable  for \eqref{redsyst}.   
\end{lemma}

We are ready to present our main statement. 

\begin{theorem}[Almost global orbital asymptotic stability] \label{thm2}
Consider the network system \eqref{ntwksyst} under the Assumptions \ref{ass3}--\ref{ass2} and such that the reduced-order dynamics  \eqref{redsyst} satisfies Assumption \ref{ass4}. Then, there exists $\sigma^f>0$, such that, for all $ \sigma \geq \sigma^f $, the networked system \eqref{ntwksyst} admits a unique nontrivial periodic orbit $\mathcal O_{1/\sigma}$, of period $\alpha_{1/\sigma}$ and  that is almost globally orbitally asymptotically stable.  Moreover,  $\mathcal O_{1/\sigma} \rightarrow \mathcal O_o$, where
\begin{equation*}
 \mathcal O_o := \{ x \in \mathbb{R}^{nN} : x_m \in \gamma_o ~\text{and} ~ e_v=0\}
\end{equation*}
 ---see \eqref{282}, and $\alpha_{1/\sigma} \to \alpha_o$ as $\sigma \rightarrow \infty$.
\end{theorem} 

The proof of Theorem \ref{thm2},  which is provided farther below,  follows a sequence of logical steps to assess the existence,  uniqueness,  and almost global orbital asymptotic stability  of an orbit for \eqref{ntwksyst}.  The analysis relies on studying the singularly-perturbed system \eqref{singpertsyst},  but we emphasize that the available literature on stability (of the origin or a compact set) for singularly-perturbed systems \cite{tikhonov},  \cite{kokotovic1976singular}, \cite{KHALIL96} does not apply to \eqref{singpertsyst},  when \eqref{redsyst} admits a limit cycle and an isolated equilibrium point.   Therefore,  the proof of Theorem \ref{thm2}  relies on technical  lemmata that are presented next,  but,  for clarity of exposition, the proofs of these lemmata are given in the Appendix.  

\begin{itemize}
\item Lemma \ref{lem1} establishes global asymptotic practical stability of the 
$$  \{ (x_m, e_v) \in \mathbb{R}^{nN} : x_m \in \gamma_o \cup \gamma_1 ~\text{and} ~ e_v=0\};  $$

\item Lemma \ref{lem2}  establishes that,  given a torus  sufficiently tight around $\mathcal{O}_o$,  for each coupling gain $\sigma$ sufficiently large,   system \eqref{redsyst} admits a unique periodic orbit $\mathcal{O}_{1/\sigma}$ contained in the torus;

\item  Lemma  \ref{lem3} establishes that each such orbit 
$\mathcal{O}_{1/\sigma}$ is (locally) asymptotically stable and admits the aforementioned torus as a basin of attraction;

\item Lemmata \ref{lem6} and \ref{clcl1} (in the Appendix)  provide  a local analysis around the origin,   to establish that it attracts only the solutions starting from a null-measure set.  
\end{itemize}

\subsection*{Technical Lemmata}

We  start by introducing the following notations.  Correspondingly to $\gamma_o \subset \mathbb{R}^n$ and $\gamma_1 \subset \mathbb{R}^n$,  which denote, respectively,  the closed periodic and homoclinic orbits for system \eqref{redsyst}---see Assumption \ref{ass4},  we introduce their  ``lifting'' 
 $\Gamma_o \subset \mathbb{R}^{nN}$ and  $\Gamma_1 \subset \mathbb{R}^{nN}$ in the space of  system \eqref{singpertsyst},  as 
\begin{align} 
   \Gamma_o : =\left\{(x_m,e_v)\in \mathbb{R}^{nN}: x_m\in\gamma_o ~ \text{and} ~ e_v=0 \right\},   \label{Gamma-o}
   \\
   \Gamma_1 := \left\{(x_m,e_v)\in \mathbb{R}^{nN}: x_m \in \gamma_1 ~ \text{and} ~ e_v=0 \right\}.  \label{Gamma-1}
\end{align}
 Furthermore, we denote by $T_\rho$ the torus defined as 
\begin{equation}  \label{trho}  
T_{\rho} :=  \left\{ \left(e_v, x_m\right) \in \mathbb{R}^{N(n-1)} \times \mathbb{R}^n \,:\, |(x_m,e_v)|_{\Gamma_o} \leq \rho \right\},
\end{equation}
and we use $\Gamma_\varepsilon \subset \mathbb{R}^{nN}$ to denote a closed orbit generated by a periodic solution to  system \eqref{singpertsyst},  {\it if it exists}. That is, $\Gamma_\varepsilon$ is a subset in the space of $(x_m,e_v)$ that consists in the image points generated by the parameterized solutions of \eqref{singpertsyst},  $t\mapsto (x_m(t),e_v(t))$,   that are  periodic with period $\alpha_\varepsilon$.

The first technical lemma provides a statement for system \eqref{singpertsyst} on   global practical asymptotic (GApS) stability of the set 
$$ \Omega := \Gamma_o \cup  \Gamma_1, $$
where $\Gamma_o$ and $\Gamma_1$ are introduced in \eqref{Gamma-o}-\eqref{Gamma-1}.

\begin{lemma} [GApS of $\Omega$]  \label{lem1}
Consider system \eqref{singpertsyst} such that Items (i)--(iii) of Lemma \ref{fact1} hold and let Assumption \ref{ass4} be satisfied for the corresponding reduced-order system \eqref{redsyst}.   
Then,   the set $\Omega := \Gamma_o \cup \Gamma_1$ is GApS for \eqref{singpertsyst}. In particular,   for any $\rho>0$,  there exists $\varepsilon_1(\rho)>0$,  such that,  for each $\varepsilon \leq \varepsilon_1(\rho)$ and for each initial condition $x_o  = (e_{vo}, z_{mo}) \in \mathbb{R}^{nN} $, the solution $(x_m(t),e_v(t))$ satisfies 
$\lim_{t \rightarrow \infty} \left| x_m(t), e_v(t) \right|_\Omega \leq \rho$.
\end{lemma}

The next lemma establishes that,  for all sufficiently small values of $\rho>0$,   there exist  sufficiently small values of $\varepsilon$,  such that there  exists a unique periodic orbit $\Gamma_\varepsilon \subset T_{\rho}$ generated by a solution to \eqref{singpertsyst} of period  $\alpha_\varepsilon \approx\alpha_o$.   

\begin{lemma}[Existence of $\Gamma_\varepsilon$] \label{lem2}
Consider system  \eqref{singpertsyst} such that  Items (i) and (iii) of Lemma \ref{fact1} hold and let Assumption \ref{ass4} hold for the reduced-order dynamics  \eqref{redsyst}.   Then,  there exist $\rho_o > 0$ and a class $\mathcal{K}$ function $\varepsilon_o$ such that,  for each $\rho \in (0,\rho_o]$ and for each 
$\varepsilon \leq \varepsilon_o(\rho)$,   system \eqref{singpertsyst}  has a unique nontrivial periodic orbit $\Gamma_\varepsilon$, which is strictly contained in $T_\rho$.   Moreover, the period $\alpha_\varepsilon$ of the solution to \eqref{singpertsyst} generating the orbit $\Gamma_\varepsilon$  tends to $\alpha_o$,  which is  the period of the solution to  \eqref{redsyst} generating   the orbit $\Gamma_o$. 
\end{lemma}

\begin{remark}
The existence result in Lemma \ref{lem2} follows from a direct application of Anosov Theorem---see Lemma \ref{thmAnosov} in the Appendix.
\end{remark}

The next lemma establishes local asymptotic orbital stability of all periodic orbits $\Gamma_\varepsilon$ lying inside the torus $T_\rho$ for sufficiently small values of $\varepsilon$ and $\rho$. Moreover, we show that the corresponding domain of attraction is uniform in $\varepsilon$.   

\begin{lemma}[Stability of $\Gamma_\varepsilon$] \label{lem3}      
Let system \eqref{singpertsyst} satisfy  Items (i) and (iii) of Lemma \ref{fact1} and let Assumption \ref{ass4} be satisfied for the corresponding reduced-order dynamics \eqref{redsyst}.  
Then,  there exist $\varepsilon^{**} > 0$ and $\rho^{**}>0$ such that,  for each $\varepsilon \leq \varepsilon^{**}$,  each periodic orbit $\Gamma_\varepsilon \subset T_{\rho^{**}}$ generated by an  
$\alpha_\varepsilon$-periodic solution to \eqref{singpertsyst}, with $\alpha_\varepsilon$ sufficiently close to $\alpha_o$,  is asymptotically orbitally stable with a domain of attraction that contains $T_{\rho^{**}}$.
\end{lemma}

\begin{remark}
Lemma \ref{lem3} is reminiscent of a statement established by Anosov---\cite[Theorem 5]{anosov1960limit}---that pertains to the case in which \color{black}the periodic orbit \color{black}$\gamma_o$  for \eqref{redsyst}  is only non-singular (or hyperbolic).  Although it is claimed in \cite{anosov1960limit} that the proof therein translates directly to the case where  $\gamma_o$  is non-singular and asymptotically stable,  in this paper,  we provide an original proof for the latter case using the theory of perturbed matrices \cite{moro1997lidskii,Stewart90}. 
\end{remark}

The next statement links those from Lemmata \ref{lem1}--\ref{lem3}.  It establishes  that if $\varepsilon $ is sufficiently small, then the periodic behavior of the reduced order system  \eqref{redsyst}  is preserved for system \eqref{singpertsyst}  as well as its stability properties.

\begin{proposition} \label{prop4}
Consider the dynamical system \eqref{singpertsyst} under the assumption that Items (i)--(iv) of Lemma \ref{fact1} hold and assume further that the reduced-order dynamics  \eqref{redsyst} satisfies Assumption \ref{ass4}.

Then,  there exists $\rho_o > 0$ such that,  for each $\rho \in (0,\rho_o]$,  there exists  $\varepsilon_2(\rho) > 0$ such that, for each $\varepsilon \in (0, \varepsilon_2(\rho)]$, 
  \begin{itemize}
    \item[(i)]  system \eqref{singpertsyst} admits  a unique orbit $\Gamma_\varepsilon\subset T_\rho$ generated by a non-trivial 
  $(\alpha_\varepsilon)$-periodic solution,  with $\Gamma_\varepsilon \rightarrow \Gamma_o$ and $\alpha_\varepsilon \rightarrow \alpha_o$ as $\varepsilon \rightarrow 0$;
\item[(ii)] $\Gamma_\varepsilon$ is (locally) asymptotically stable;
\item[(iii)] for any initial condition $\bar x_o \in \mathbb{R}^{nN}$ the corresponding solution $\bar x(t)$ to \eqref{singpertsyst} either converges to $\Gamma_\varepsilon$ or to a $\rho$-neighborhood of $\Gamma_1$, that is,    
\begin{equation}
  \label{664} 
  \limsup_{t \rightarrow \infty} \left|  \bar x(t) \right|_{\Gamma_1} \leq \rho.
\end{equation}
  \end{itemize}
\end{proposition}

\begin{proof}
Items (i)-(iv)  in Lemma \ref{fact1}  and Assumption \ref{ass4} imply that  the statements of Lemmata \ref{lem1}--\ref{lem3} hold.  Then,  let Lemma \ref{lem3} generate $\varepsilon^{**}>0$ and $\rho^{**}>0$.  Furthermore, let Lemma \ref{lem2} generate $(\rho_o, \varepsilon_o(\cdot))$,  for each $\rho \in (0,  \rho_o]$, let Lemma \ref{lem1} generate $(\varepsilon_1(\rho), \varepsilon_1(\rho^{**}))$, and let 
$$ 
\varepsilon \leq \varepsilon_2(\rho) := \min \{   \varepsilon^{**},  \varepsilon_o(\rho),   \varepsilon_o(\rho^{**}),  \varepsilon_1(\rho),  \varepsilon_1(\rho^{**})   \} $$ 
be arbitrarily fixed. 

After Lemma \ref{lem2},   there exists a unique periodic orbit $\Gamma_{\varepsilon} \subset T_{\rho} \cap T_{\rho^{**}} =  T_{\min\{\rho, \rho^{**} \}}  $ generated by a solution to \eqref{singpertsyst}.  Now, given any sequence $\{ \epsilon_i\}^{\infty}_{i=1}$ that converges to zero and such that $\epsilon_i \leq \varepsilon_2(\rho)$ for all $i \in \{1,2,... \}$,  from the above, we know that,  for $i$ large enough,  the unique orbit $\Gamma_{\epsilon_i}$ satisfies $\Gamma_{\epsilon_i} \subset T_{\rho_i}$, where  $\rho_i := \varepsilon_o^{-1}(\epsilon_i)$---note that  $\varepsilon_o^{-1}$ exists and is of class $\mathcal{K}$ because so it $\varepsilon_o$.   
Item [i] of the proposition follows since $T_{\rho_i}$ converges to $\Gamma_o$---see \eqref{trho}.

Next, after Lemma \ref{lem3},  we conclude that  $\Gamma_{\varepsilon}$ is orbitally asymptotically stable  and $T_{\rho^{**}}$ is inside the domain of attraction of $\Gamma_{\varepsilon}$.  This establishes Item (ii). 

Finally, from Lemma \ref{lem1}  we conclude that each solution to \eqref{singpertsyst} either converges to $T_{\min\{\rho, \rho^{**} \}} \subset T_{\rho^{**}}$,  so it also converges to $\Gamma_{\varepsilon}$,  or it converges to a  $\min\{ \rho, \rho^{**} \}-$neighborhood 
of $\Gamma_1$.  This establishes Item (iii).
\end{proof}

The last technical lemma  provides  a local stability analysis around the origin  of \eqref{singpertsyst}.  It states that the origin is a hyperbolic equilibrium point, for all sufficiently-small values of $\varepsilon$.  Furthermore,  inspired by the Stable Manifold Theorem \cite[Theorem~13.4.1]{coddington1955theory},  we show that the stable and unstable manifolds around the origin are uniquely defined on a neighborhood whose size does not shrink with $\varepsilon$.   

 \begin{lemma} [Local behavior around the origin]  \label{lem6}
 Consider system \eqref{singpertsyst} and let Items (i)-(ii) of Lemma \ref{fact1} hold.  Assume further that the corresponding reduced-order dynamics  \eqref{redsyst} satisfies Assumption \ref{ass4}.  Then,  there exist $\rho^*>0$, $\varepsilon^* >0$,  a neighborhood of the origin denoted  $U  \subset  \mathbb{R}^n$,    and $r > 0$ such that,  for each 
$ \varepsilon \in (0,  \varepsilon^*)$,  
\begin{itemize}
\item[(i)] system \eqref{singpertsyst} admits a unique unstable and stable manifolds $(W^u_\varepsilon(0), W^s_\varepsilon(0))$ defined on $U$;
\item[(ii)] for each $\bar{x}(t)$ bounded solution to \eqref{singpertsyst} starting from $\bar{x}_o \in U \backslash W^s_\varepsilon(0)$,  there exists $T_1 > 0$ such that  $|\bar{x} (t)| \geq r$  for all $t \geq T_1$;
\item[(iii)] for each $\bar{x}(t)$ solution to \eqref{singpertsyst} such that $|\bar{x}(0)|_{\Gamma_1} \leq \rho^*$,  there exists $T_2 > 0$ such that $|\bar{x} (T_2)| < r$.   
\end{itemize} 
\end{lemma}

\subsection*{Proof of Theorem \ref{thm2}}

Under Assumptions \ref{ass3}--\ref{ass2},  Items (i)-(iv) of Lemma \ref{fact1} hold for system \eqref{singpertsyst}.  This and  Assumption \ref{ass4} imply that  the statements of  Proposition \ref{prop4}  and Lemmata \ref{lem1}--\ref{lem6} hold.  Then, let Lemma \ref{lem6} generate $\rho^{*}>0$ and $\varepsilon^*>0$ and let Proposition \ref{prop4} generate $\left(\rho_o, \varepsilon_2(\min\{\rho^*, \rho_o\}/4)\right)$.  We show that the statement of Theorem \ref{thm2} holds with $\sigma^f := 1 / \min \left\{  \varepsilon^*,  \varepsilon_2(\min\{\rho^*, \rho_o\}/4) 
 \right\} $ in the following four steps: 

\begin{enumerate}

\item First, for any  $\sigma\geq \sigma_f$ or, equivalently, any $\varepsilon = 1/\sigma$ satisfying $\varepsilon \leq \min \left\{  \varepsilon^*,   \varepsilon_2(\min\{\rho^*, \rho_o\}/4)  \right\}$,   we use Item (i) of Proposition \ref{prop4} to conclude the existence of a unique nontrivial periodic orbit $\Gamma_\varepsilon$ generated by a periodic solution to \eqref{singpertsyst} of period $\alpha_{\varepsilon}$. From Item (ii) of the same Proposition it follows that $\Gamma_\varepsilon$ is locally asymptotically stable. In addition, from Item (iii) of Proposition \ref{prop4} it follows that  each solution $\bar x(t)$ to system \eqref{singpertsyst} either converges to the orbit $\Gamma_\varepsilon$, otherwise,  it converges to a 
$(\min\{\rho^*, \rho_o\}/4)$--neighborhood of $\Gamma_1$;  that is,  \eqref{664} holds with $\rho =  \min\{\rho^*, \rho_o\}/4 $ and, consequently,  there exists $T<\infty$ such that 
\begin{align} \label{eqrhostar}
|\bar{x}(t)|_{\Gamma_1}  =  |(x_m(t), e_v(t))|_{\Gamma_1}  \leq \min\{\rho^*, \rho_o\}/2 \qquad \forall t \geq T.   
\end{align}

\item Let us now introduce the backward propagation of the stable manifold $W^s_\varepsilon(0)$ introduced for \eqref{singpertsyst} in Lemma \ref{lem6}. That is,  we  introduce  set 

\[
 R(W^s_\varepsilon(0))  :=   \{  \bar{x}(t) :  t \leq 0,~ \bar{x}(0) \in W^s_\varepsilon(0)
\}   
\]
and prove by contradiction that \color{black} the solution $\bar{x}$ to \eqref{singpertsyst} satisfying \eqref{eqrhostar} must start from the set $R(W^{s}_\varepsilon(0))$.   Indeed,  assume that the opposite holds.  Then,  using Item (iii) in Lemma \ref{lem6},  we conclude that the solution $\bar{x}$ must enter $B_r$ at some $T^*  \geq T$.  In particular,  $\bar{x}(T^*) \in U$ and $\bar{x}(T^{*}) \notin W^{s}_\varepsilon(0)$.  Now,  using  Item (ii) in  Lemma \ref{lem6}, we conclude that there exists $T_1>0$ such that 
$$ | \bar{x}(T^* + t) | \geq r \qquad \forall t \geq T_1.  $$
However, since $ | \bar{x}(T^* + T_1) |_{\Gamma_1} \leq \min\{\rho^*, \rho_o\}/2$, it follows that $\bar{x}$ must enter $B_r$ again under Item (iii) of Lemma \ref{lem6},  which contradicts Item (ii). 

\item Next,  we show that the set $R(W^s_\varepsilon(0))$ is a null measure set using contradiction.   That is,  let $S_o \subset R(W^s_\varepsilon(0))$  such that $\mu (S_o) \neq 0$. Assume without loss of generality that for some $T < 0$,  we have 
\begin{align*}
& S_o \subset 
\\ &  \{  \bar{x}(t) :  t \in [-T, 0],~ \bar{x}(0) \in W^s_\varepsilon(0), ~ \bar{x} ~\text{solution to \eqref{singpertsyst}} \}.  
\end{align*}
Note that
$$  R^b(T,S_o) := \{ \bar{x} (T) : \bar{x}(0) \in S_o  \}   \subset W^s_\varepsilon(0) $$ 
with  $\mu(W^s_\varepsilon(0)) = 0$.  However,  using Lemma \ref{clcl1} from the Appendix,  we conclude that  $\mu (S_o) = 0$. 

\item Finally,  using the inverse transformation \eqref{invbarx},  it  follows that the orbit 
$$ \mathcal O_\varepsilon := \{  x \in \mathbb{R}^{nN} : (x_e,e_v) \in \Gamma_\varepsilon\} $$ 
is  almost GAS for  \eqref{ntwksyst}.  The second statement follows from Lemma \ref{lem2}. 
\end{enumerate}
\hfill $\blacksquare$


\section{Case study: a network of Andronov-Hopf oscillators}\label{sec:casestudy}
\label{A-H}

For illustration of our main theoretical findings, we address, as a case-study, the analysis of a network of Andronov-Hopf oscillators. The equation of such oscillator represents a normal form of the bifurcation carrying the same name and is given by 
\begin{equation}\label{eq:1.001ex}
\dot  z = - \nu |z|^2 z + \mu z,
 \end{equation}
where $z \in \mathbb{C}$ denotes the state of the oscillator and $\nu$, $\mu\in\mathbb{C}$ are constant parameters; $\nu:= \nu_R + j \nu_I$ and $\mu:=\mu_R + j \mu_I$.  The analysis of \eqref{eq:1.001ex} is well documented in the literature. For instance, via Lyapunov-exponents-based methods, as in  \cite{Kuznetsov} and \cite{Perko}, or using  Lyapunov's direct method, as in \cite{Strogatz2} and \cite{slot}. The behavior on the phase plane is illustrated in Fig. \ref{ffigHopf}. 

\begin{figure}[h!]
\begin{center}
  \begin{minipage}{4in}
    \begin{picture}(0,130)
    \put(55,-32){\includegraphics[width=6cm,height=6cm]{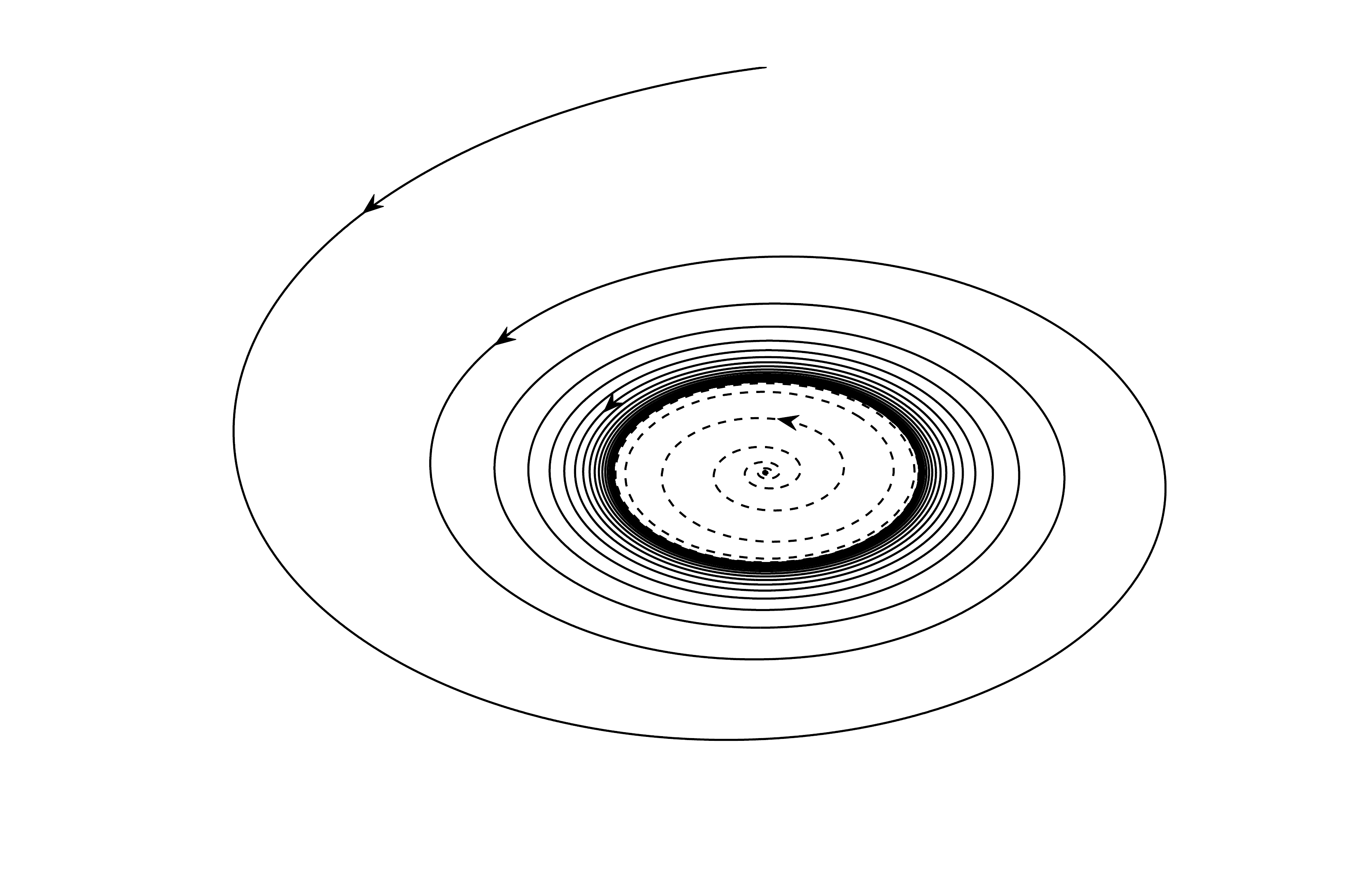}}
  \end{picture}
\caption{Trajectories of the Andronov-Hopf oscillator on the complex plane when $\mu_R > 0$.  In this case,  the origin is an unstable equilibrium and all trajectories tend to the stable limit-cycle of radius $\sqrt{\mu_R}$ and period $\frac{2 \pi}{\mu_I}$.}
\label{ffigHopf}
  \end{minipage}
\end{center}
\end{figure}  

Now, for the purpose of this paper, we consider $N$ forced  Andronov-Hopf oscillators for which, for simplicity, we assume that $\nu = 1$,  the general case where $\nu \neq 1$ will be considered elsewhere.  That is, consider the systems 
\begin{align} \label{AH-unitex}
\dot{z}_i = -|z_i|^2 z_i + \mu_i z_i + u_i \qquad i \in \{1,2,\ldots,N\}
\end{align}
where $z_i \in \mathbb{C}$ denotes the $i$-th oscillator's state, $u_i \in \mathbb{C}$ is its control input,  and $\mu_i := \mu_{Ri} + i \mu_{Ii}$ is a constant complex parameter.   

Now we design each input $u_i$ as in \eqref{ui} while replacing $(x_i,x_j)$ therein by $(z_i,z_j)$.  Then,  in compact form, the overall network dynamics in closed loop takes the form 
\begin{equation} \label{ntwksystex}
\dot  z =  F({ z})  -\sigma 
L z, 
\end{equation}
where $z := [z_1\,\cdots\,z_N]^\top$ and the function 
$F: \mathbb{C}^N \to \mathbb{C}^N$ is given by 
\begin{equation*} 
\begin{aligned}
F(z) & :=  \left[f_1(z_1) \quad f_2(z_2) \quad \cdots \quad  f_N(z_N)\right]^\top, 
\\ 
f_i(z_i)& := - |z_i|^2 z_i + \mu_i z_i. 
\end{aligned}
\end{equation*}

Next,  using  \eqref{ntwksystex} and the coordinate transformation defined in  \eqref{282},  we can rewrite the  network dynamics in the singular perturbation form,  where the reduced-order dynamics  defined by \eqref{redsyst} has the form 
\begin{align} \label{A13ex}
\dot{z}_m =  f_m(z_m) := -|z_m|^2 z_m + \mu_m z_m,
\end{align}
where $z_m = v_l^\top z$,  the parameter $\mu_m \in \mathbb{C}$  corresponds to the weighted average of the $\mu_i$s, {\em i.e.}, 
\begin{equation*}
\begin{aligned} 
\mu_m & := \mu_{mR} + i \mu_{mI}, 
\\ 
\mu_{mR}  & := \sum^N_{j=1} v_{lj} \mu_{Rj},   \quad 
 \mu_{mI}  := \sum^N_{j=1} v_{lj} \mu_{Ij}.
\end{aligned}
\end{equation*}

Hence,  the reduced-order dynamics \eqref{A13ex} is an Andronov-Hopf oscillator, but we note that $\mu_{mR}$ and $\mu_{mI}$ do not correspond to simple averages of the individual parameters $\mu_{Ri}$ and $\mu_{Ii}$; through the Laplacian's eigenvector $v_l$, they depend on  the network's topology. 

Two cases that are of interest are studied next.  
 
\subsection{Case in which  $\mu_{mR} \neq 0$}

The parameter $\mu_{mR}$, which depends both on the individual parameters of the oscillators and on the structure of the network, is particularly important.  Indeed, as we shall show,  for $\sigma$ sufficiently large, the  asymptotic collective behavior of the network is uniquely defined by the sign of $\mu_{mR}$.   
\begin{proposition}  \label{prop2}
Consider the network of Andronov-Hopf oscillators defined in \eqref{ntwksystex} interconnected over a connected directed graph. Then, the following statements hold true:
\begin{itemize}
\item[(i)] If $\mu_{mR} < 0$ then there exists  $\sigma^f>0$ such that, 
 for all $\sigma\geq \sigma^f$,  the origin $\{z = 0\}$ is GAS. 

\item[(ii)]  If $\mu_{mR} >0$ then there exists $\sigma^f>0$ such that, for each $\sigma\geq \sigma^f$,  system \eqref{ntwksystex} has a unique nontrivial periodic orbit which is almost globally asymptotically stable. \\[-20pt]
\end{itemize} 
\end{proposition}

\begin{proof}
  The network model \eqref{ntwksystex} satisfies Assumptions \ref{ass3}--\ref{ass2} regardless of the sign of $\mu_{mR}$. Indeed,  Assumption \ref{ass3} is explicit in the Proposition. 
On the other hand,  each single unit has a polynomial vector field, so Assumption \ref{ass1} holds. To see that Assumption \ref{ass2} on semi-passivity of each unit in \eqref{AH-unitex} holds, we evaluate the total derivative of  the Lyapunov function candidate 
$$V(z_i) := \frac{1}{2} z_i^* z_i,  $$ 
where $z_i^*$ is the complex conjugate of $z_i$, 
to obtain
$$ \dot{V}(z_i) = - \left|z_i\right|^4 + \mu_{Ri} \left|z_i\right|^2 + z_i^* u_i.   $$    

 Now, to complete the proof, we note that if  $\mu_{mR} < 0$,  it follows that $\dot{V}(z_m)$ is negative definite and the origin for  the reduced-order dynamics  \eqref{A13ex} is globally asymptotically stable. Hence,  after Theorem \ref{thm:gas}, Item (i) of the proposition follows. 

In the opposite case,  if  $\mu_{mR} > 0$,  the total derivative of $V(z_m)$ along the solutions to \eqref{A13ex} satisfies 
\begin{align}  \label{eqdotVzm}
\dot{V}(z_m) = - \left|z_m\right|^4 + \mu_{mR} \left|z_m\right|^2. 
\end{align} 
Hence,   system \eqref{A13ex} admits a compact invariant set that is composed of two disjoint invariant subsets, that is
\begin{equation*}
\omega =  \left\{ z_m \in \mathbb{C}: \, |z_m| =  \sqrt{\mu_{mR}}\, \right\}  \cup \big\{z_m = 0\big\}.
\end{equation*}
Moreover, the invariant orbit 
$$ \gamma_o := \left\{ z_m \in \mathbb{C} : \, |z_m| =  \sqrt{\mu_{mR}}\, \right\} $$  
is almost globally asymptotically stable and the origin 
$\{z_m  = 0\}$ is anti-stable  ---  see~\cite{MICNON-2015}.  
As a result,  Item (ii) of the proposition follows using Theorem \ref{thm2}. 
\end{proof}

\subsection{Case in which $\mu_{mR} = 0$}

When $\mu_{mR}=0$,  we  use \eqref{eqdotVzm} to conclude that the reduced-order dynamics  \eqref{A13ex}  admits the origin as a GAS equilibrium.  Hence,  using Theorem \ref{thm:gas},  we conclude  GApS of the origin for \eqref{ntwksystex}.   

Now,  we would like to answer the following two questions:

\begin{enumerate}
 \item Are there cases where,  although being  globally practically asymptotically stable,  the origin for  \eqref{ntwksystex} is locally unstable?

\item Under which conditions,   the origin for \eqref{ntwksystex} is locally exponentially stable?  

\item Are we able, to establish global exponential stability?
\end{enumerate}

Consider the linearization of \eqref{ntwksystex} around the origin, which is given by 
\begin{align} \label{eqLin}
\dot{z} = A_\sigma z,  ~ A_\sigma := A_o  - \sigma  L,  ~  A_o := \text{blkdiag} \left\{ \mu_{1}, \mu_2,...,\mu_N   \right\}.
\end{align}

To answer the first question,  in the following example,  we provide a case where the matrix $A_\sigma$ has an eigenvalue with a strictly positive real part for all  $\sigma >0$ sufficiently large.  Hence, we provide a positive answer to the question in 1).   

\begin{example}  \label{expmu}
 Consider a network of two Andronov-Hopf oscillators, $N=2$,   and let $\mu_{R1} = -\mu_{R2} = \mu$, and $\mu_{I1} = \mu_{I2} = 0$.  
 Clearly,  in this case,   $\mu_{mR}=0$.   Note that in this particular case
\begin{align*}
 A_{\sigma} & =  \begin{bmatrix} \mu   & 0\\ 0 &  -\mu  \end{bmatrix} -\sigma L.  
\end{align*}  

Next, we let for simplicity $\mu:=1$ and  $L := \begin{bmatrix} 1  & -1 \\ -1 & 1 \end{bmatrix}$,   so 
\begin{align*}
A_\sigma & 
=  \begin{bmatrix} -\sigma + 1 & \sigma \\ \sigma & - \sigma - 1 \end{bmatrix}. 
\end{align*}  

Direct calculations show that the eigenvalues of the matrix $A_\varepsilon$ are given by  
$$\lambda_{1,2}(A_\sigma) = -\sigma \mp \sqrt{ \sigma^2 + 1}$$ and therefore the matrix $A_\sigma$ has 2 positive eigenvalues for all $\sigma>0$.  Hence,  linearization of the network dynamics is unstable and therefore (see Theorem 4.7, Khalil, 3rd) the origin is unstable for the network of Andronov-Hopf oscillators. 
\end{example}

Next,  to answer the second question,  we consider the following assumption.
\begin{assumption} \label{cl1}
There exists $\sigma^\star > 0$ such that, for all $\sigma \geq \sigma^\star$,  the smallest (in norm) eigenvalue of $A_\sigma$,  denoted by $\lambda_{1} (A_\sigma)$, has a strictly negative real part. 
\end{assumption}

When Assumption \ref{cl1} holds,  we show that the origin for \eqref{ntwksystex} is  locally exponentially stable.  

\begin{proposition} \label{prop3}
Consider the network of Andronov-Hopf oscillators  in \eqref{ntwksystex},  interconnected over a connected directed graph, and such that Assumption \ref{cl1} holds and $\mu_{mR} = 0$.  Then,  there exists  $\sigma^f>0$ such that,  for all 
$\sigma\geq \sigma^f$,  the origin $\{z = 0\}$ is locally exponentially stable. 
\end{proposition}

In the following example, we modify Example \ref{expmu} so that Assumption \ref{cl1} holds. 

 \begin{example}
  Consider a network of two Andronov-Hopf oscillators, $N=2$,   and let $\mu_{R1} = -\mu_{R2} = 1$, and $\mu_{I1} = - \mu_{I2} = 2$.  
 Clearly,  in this case,   $\mu_{mR}=0$.   Note that in this particular case
\begin{align*}
 A_{\sigma} & =  \begin{bmatrix} 1 - 2 i   & 0\\ 0 &  -1 + 2 i  \end{bmatrix} -\sigma L.  
\end{align*}  
Next, we let  $L := \begin{bmatrix} 1  & -1 \\ -1 & 1 \end{bmatrix}$ to obtain 
\begin{align*}
A_\sigma 
=  \begin{bmatrix} -\sigma + 1 - 2i & \sigma \\ \sigma & - \sigma - 1 + 2i \end{bmatrix}. 
\end{align*}  
Direct calculations show that the characteristic polynomial of 
$A_\sigma$ is given by 
$$  P(\lambda) := \lambda^2  + 2 \sigma \lambda + (3 + 4i).   $$

The roots of the polynomial $P$  are the eigenvalues of the complex matrix $A_\sigma$ and are given by
$$  \lambda_{1,2} := - \sigma \pm \sqrt{\sigma^2 - 3 - 4i}.  $$
As a result,  the real parts of $ \lambda_{1,2} $ are given by 
$$ \Re(\lambda_{1,2}) := - \sigma \pm \frac{  \sqrt{ (\sigma^2 - 3) + \sqrt{ (\sigma^2 - 3)^2 + 16   }    }     }{\sqrt{2}}.   $$
Clearly,  $\Re(\lambda_{1,2})$ are strictly negative for $\sigma > 0$ sufficiently large.   Hence,  linearization of the network dynamics is exponentially stable. 
 \end{example}

The third question is an open question to the best of our knowledge.  Indeed,  since we already know that the origin for \eqref{ntwksystex} is GApS,  one way to give a positive answer for the third question is by showing  that the origin for \eqref{ntwksystex},   which  is already locally exponentially stable under  Assumption \ref{cl1},  has a basin of attraction that does not shrink as $\sigma$ get bigger or shrinks at a specific rate.   However, we are not, so far, able to establish this property.


\section{Conclusion}  \label{sec:concl}

We presented a framework to analyse networks of heterogeneous nonlinear systems. Our approach allows to qualitatevely characterize the collective behavior for ``large'' values of the coupling gains.  The proposed approach, however, does not give much information about the quantitative properties of such behavior.   In particular,  we do not provide explicit values of the coupling strength under which the networked system exhibit the established behavior.   Characterizing the emergent behavior,  such as orbital asymptotic stability,  both qualitatively and  quantitatively (in terms of the coupling gain) is still an open problem.  
Furthermore,  beyond the analysis problems solved in this paper,  the control design problem is widely open.  To find conditions under which a network of heterogeneous systems may be controlled so that it admits a {\it desired} reduced-order dynamics.  
Finally,  we believe that extending the proposed framework for general classes of nonlinear systems such as hybrid systems  is an interesting perspective as well.

\section*{Appendix I: Background}

\subsection*{Distortion of the solutions to singularly perturbed systems}

We recall a version of the well-known Tikhonov Theorem   \cite[Theorem 9.1]{KHALIL96}, originally published in \cite{tikhonov},   which allows to approximate the trajectories of a singularly perturbed system by the trajectories of its unperturbed form,  along a finite interval of time. 
 
\begin{lemma}[Tikhonov Theorem]  \label{thmtykho} 
Consider the singularly perturbed system 
\begin{equation}
\label{eqsingper}
\begin{aligned} 
\dot{z} & =  f(z,e,\varepsilon) \\
\varepsilon \dot{e} & =  g(z,e,\varepsilon)   \qquad (z,e) \in  \mathbb{R}^{m_z} \times \mathbb{R}^{m_e}.
\end{aligned}
\end{equation}
  Assume that there exist  $t_1$, $r$,  $\rho$,  and 
$\varepsilon_o > 0$ such that, for each $(z,e,\varepsilon) \in  B_r \times B_\rho \times [0,\varepsilon_o]$, we have 
\begin{itemize}
\item the functions $f$, $g$, and their first partial derivatives with respect to $(z,e,\varepsilon)$ are continuous. 
\item the function $e = h(z)$, solution to the equation $0=g(z,e,0)$, and the Jacobian matrix $ \frac{\partial g(z,e,0)}{\partial e^\top}$ have continuous first partial derivatives with respect to their arguments.
\item the unperturbed system  
\begin{align} \label{eqbarz}
\dot{\bar z} = f(\bar z , h(\bar z),0) 
\end{align}
has unique solutions on $[0,t_1]$ when starting from $B_r$. Furthermore,  there exists $r_1 \geq  r > 0$ such that, for each initial condition $\bar{z}_o \in B_{r}$,  the solution $\bar{z}(t)$ to \eqref{eqbarz} starting from $\bar{z}_o$ satisfies    
$$ \left|\bar z(t) \right| \leq r_1  \qquad  \forall t \in [0, t_1].  $$
\item the origin of the boundary-layer system 
\begin{align} \label{eqbly}
y'= g(z,y+h(z),0) 
\end{align}
is exponentially stable, uniformly in $z$. 
\end{itemize}

Then,  there exist $\mu^*>0$,  $\varepsilon^* \in (0,\varepsilon_o]$,  and $M > 0$ such that,  
for each  $\varepsilon \in (0, \varepsilon^*)$,  the solution $(z(t),e(t))$ to \eqref{eqsingper} starting from $(z_o,e_o) \in B_r \times B_\rho$,  the solution $\bar{z}(t)$ to \eqref{eqbarz} starting from 
$z_o$, and the solution $y(t)$ to \eqref{eqbly} starting from  $e_o - h(z_o)$ with $ \left| e_o - h(z_o) \right| \leq \mu^*$  satisfy,  for each $t \in [0 , t_1]$,
\begin{equation}  \label{eqbound}
\begin{aligned}
	\left| z(t) - \bar z(t) \right| & \leq  M \varepsilon \nonumber \\
	\left| e(t) - h(\bar z(t)) - y\left( \frac{t}{\varepsilon} \right) \right| & \leq M \varepsilon.
\end{aligned}
\end{equation}
\end{lemma}

\begin{remark} \label{RemThyk}
Note that system \eqref{singpertsyst} verifies all smoothness conditions in the first item of Lemma \ref{thmtykho} after Item (i) in Lemma \ref{fact1}.  We also know that the trajectories of both \eqref{singpertsyst} and \eqref{redsyst} exist and are uniquely defined on the interval $[0, \infty]$.   Furthermore,  the boundary-layer model \eqref{eqbly},  in the case of \eqref{singpertsyst},  is given by the linear system
 $ \dot{e}_v = - ( \Lambda \otimes I_n) e_v,  $
 which is exponentially stable under Item (iv) in Lemma \ref{fact1}.  Finally,  under lemma \ref{lem0},  we know the existence of $r \geq 0$ such that the set 
$$ B_r:= \left\{ (x_m,e_v) \in \mathbb{R}^n \times \mathbb{R}^{n(N-1)} : \left|(x_m,e_v)\right| \leq r \right\} $$
is globally attractive.  Thus,  all the Items in Lemma \ref{thmtykho} are satisfied.  
\end{remark}

\subsection*{Existence of periodic orbits for singularly perturbed systems}

The following result,  which is a consequence of  the main statements in \cite{anosov1960limit}, establishes the existence of periodic solutions for singularly-perturbed systems.

\begin{lemma} \label{thmAnosov}
Consider the singularly perturbed system \eqref{eqsingper} such that the following properties hold:
\begin{enumerate}
\item the functions $f$ and $g$ are continuous with respect to $(z,e,\varepsilon)$  and differentiable with respect to $z$ and $e$.  Moreover, the derivatives of $f$ and $g$ with respect to $z$ and $e$ depend continuously on  $(z,e,\varepsilon)$. 
\item there is a unique function  $h : \mathbb{R}^{m_z} \rightarrow \mathbb{R}^{m_e}$ such that $g(z,h(z),0)=0$.
\item the equilibrium state $ y = 0$ of \eqref{eqbly} is hyperbolic uniformly in $z$.
\item the unperturbed system \eqref{eqbarz} has a nontrivial  nonsingular periodic orbit $\gamma_o \subset \mathbb{R}^{m_z}$. 
\end{enumerate} 

Then,  there exists $\rho_o > 0$ and a class $\mathcal{K}$ function 
$\varepsilon_o$ such that for each $\rho \in (0,\rho_o]$ and for each $\varepsilon \leq \varepsilon_o(\rho)$,  the system \eqref{eqsingper}  has a unique nontrivial periodic orbit $\Gamma_\varepsilon$,  which is strictly contained in the $\rho$-neighborhood of $\Gamma_o$, where 
$$ \Gamma_o := \{ (z,e) \in \mathbb{R}^{m_z} \times \mathbb{R}^{m_e} :  z \in \gamma_o ~ \text{and} ~ e = h(z) \}.  $$
 Moreover, the period $\alpha_\varepsilon$ of the periodic solution to \eqref{eqsingper} generating the orbit $\Gamma_\varepsilon$  tends to $\alpha_o$ the period of the solution to  \eqref{eqbarz} generating   the orbit $\Gamma_o$. 
\end{lemma}
 
\subsection*{Input-to-state stability of compact disconnected subsets} 

  In this section,  we recall a version of \cite[Theorem 1]{angeli2015characterizations} on the characterization of input-to-state stability of compact disconnected subsets for nonlinear systems of the form 
\begin{align} \label{eqSysIss} 
\dot{x} = f(x,d) \qquad (x,d) \in \mathbb{R}^n \times D,  
\end{align}
where $D \subset \mathbb{R}^{m_d}$ is a closed subset and $f : \mathbb{R}^n \times D \rightarrow \mathbb{R}^n$ is continuously differentiable.   Consider a compact subset  $\omega \subset \mathbb{R}^n$ that is invariant for the unperturbed system 
\begin{align} \label{eqSysIss1}
\dot{x} = f(x,0).
\end{align}

The attracting and repulsing subsets of $\omega$ are given by
\begin{align*}
U(\omega) := \{ x_o \in \mathbb{R}^n :  \lim_{t \rightarrow + \infty} 
|x(t)|_\omega  = 0 \},
\\
R(\omega) := \{  x_o \in \mathbb{R}^n :  \lim_{t \rightarrow - \infty} 
|x(t)|_\omega  = 0 \}.
\end{align*}
Furthermore,  we consider the following assumption.

\begin{assumption} \label{assdisco-}
The set $\omega$ is the union of disconnected and nonempty subsets $\{ \omega_i \}^k_{i=1} \subset \mathbb{R}^n$, for some $k \in \mathbb{N}$.  
\end{assumption}

\begin{definition}
We say that $\omega_i < \omega_j$,  for some  $i$, $j \in \{1,2,...,k \}$,  if $U(\omega_i) \cap R(\omega_j) \neq \emptyset$, which implies the existence of a solution to \eqref{eqSysIss1} relating $\omega_i$ to $\omega_j$.  
\end{definition}

 Next, we introduce the following  properties:
\begin{itemize}
\item An $r$-cycle $(r \in \{2,3,...,k\})$ is an ordered $r$-tuple of distinct indices  $i_1,i_2,...,i_r \subset \{1,2,...,k\}$ such that $\omega_{i_1} < \omega_{i_2} < ....  < \omega_{i_r} < \omega_{i_1}$. 
\item A 1-cycle is an index $i \in \{1,2,...,k\}$ such that $[ R(\omega_i) \cap U(\omega_i) ] - \omega_i \neq \emptyset$.
\item A filtration ordering is a numbering of the $\omega_i$s so that $\omega_i < \omega_j$ implies $i \leq j$. 
\end{itemize}

The result in \cite[Theorem 1]{angeli2015characterizations} is based on the following assumption:

\begin{assumption} \label{assdisco}
The decomposition of $\omega$ into $\omega := \bigcup^{k}_{i=1}  \omega_i$ has no cycles and  forms a filtration ordering of $\omega$. 
\end{assumption} 

\begin{lemma} [Angeli-Effimov Theorem] \label{lemAE}
Consider the nonlinear system in \eqref{eqSysIss} and let $\omega$ be a compact subset that is invariant for the unperturbed system 
$\dot{x} = f(x,0)$ such that Assumptions \ref{assdisco-}-\ref{assdisco} hold.   Then,   the following two properties are equivalent.

\begin{itemize}
\item The system enjoys the asymptotic gain property; namely,   there exists a class $\mathcal{K}_\infty$ function $\eta$ such that each solution $x$ to \eqref{eqSysIss} satisfies 
$$ \limsup_{t \rightarrow + \infty} |x(t)|_\omega \leq \eta(|d|_\infty).  $$
\item The system admits an input-to-state Lyapunov function; namely,   there exists a $\mathcal{C}^1$ function $ V : \mathbb{R}^n \rightarrow \mathbb{R}_{\geq 0}$,   class $\mathcal{K}_{\infty}$ functions $\alpha, \bar{\alpha}, \underline{\alpha},  \gamma$, and a positive constant $c \geq 0$, such that
$$ \underline{\alpha}\left( \left|x\right|_{\omega} \right) \leq V(x) \leq \bar{\alpha} \left( \left|x\right|_{\omega}+c \right)  $$  
and
$$  \frac{\partial V}{\partial x^\top}(x) F(x,d)  \leq - \alpha 
\left( \left| x \right|_{\omega} \right) + \gamma (|d|)  \qquad \forall x \in \mathbb{R}^n.   $$
\end{itemize}
\end{lemma}

\subsection*{Eigenvalues of linearly perturbed matrices}

In this section,  we recall the general result due to 
Lidskii-Vishik-Lyusternik \cite[Theorem~2.1]{moro1997lidskii},  where the first order expansion of the eigenvalues of the matrix $(A +\varepsilon B)$,  with respect to $\varepsilon$,  and  the explicit formulas for the leading coefficients are given.   Before recalling the result,  we start introducing few notions.  Given a complex matrix $A \in \mathbb{C}^{n \times n}$,   we assume without loss of generality that $A$ has a single multiple eigenvalue $\lambda \in \mathbb{C}$. 
Hence, the Jordan form of $A$ is given by 
$$  J:= \text{blkdiag} \left\{ J_{11},  ..., J_{1r_1}, ......, J_{q1},  ..., J_{qr_q} \right\}.  $$ 

Note that, for each $j \in \{1,2,...,q\}$,  the blocks $\{J_{ji}  \}^{r_j}_{i=1}$ have the same size denoted $n_j$.  We assume without loss of generality that $n_1 < n_2 < ... < n_q$.  Furthermore,  for each $J_{ji}$,  $i \in \{1,2,...,r_j\}$ and $j \in \{1,2,....,q\}$, distinct right and left eigenvectors are associated,  which we denoted by $p_{ji} \in \mathbb{C}^{n}$ and $q_{ji}^\top  \in \mathbb{C}^{n}$,  respectively.  To know  which right and left eigenvectors are associated to the block $J_{ji}$, we use the nonsingular matrices $T$,  $U \in \mathbb{R}^n$ such that 
$ J =  U A  T$,  and use the decompositions 
\begin{align*}
T & =  [  T_{11},  ..., T_{1r_1}, ......, T_{q1},  ..., T_{qr_q} ],  
\\ 
U^\top & = [ U^\top_{11}, ... , U^\top_{1r_1}, ...... , U^\top_{q1},  ...,  U^\top_{qr_q}].  
 \end{align*}
As a result,  we pick $p_{ji}$ as the first column of $T_{ji}$ and $q_{ji}$ as the last row $U_{ji}$.

Now,  Collecting the aforementioned vectors allows us to introduce the following matrices 
$$   P_j := [p_{j1}, ..., p_{jr_j}],   
 \quad Q_j := \begin{bmatrix} q_{j1} \\  \vdots \\ p_{jr_j} \end{bmatrix} \qquad j \in \{1,2,..., q\} $$ 
and 
$$   \mathcal{P}_s := [P_{1}, ..., P_{s}],   
 \quad \mathcal{Q}_s := \begin{bmatrix} Q_{1} \\  \vdots \\ Q_{s} \end{bmatrix} \qquad s \in \{1,2,..., q\}.  $$ 
Finally, we let $  E_1 := I$  and $E_s := 
\begin{bmatrix} 0 & 0 \\ 0 & I \end{bmatrix}$ for all $s \in \{2,3,...,q\}$.

\begin{lemma}[Lidskii Theorem] \label{lemregpert}
Given $j \in \{1,2,...,q\}$,  if the matrix $\mathcal{Q}_{j-1}  B \mathcal{P}_{j-1}$ is nonsingular.  Then, there are $r_j n_j$ eigenvalues of $A + \varepsilon B$ admitting a first order expansion 
\begin{align*}
\lambda^{l}_{jk}(A + \varepsilon B)  & = \lambda + \zeta_{jk}^{\frac{1}{n_j}} \varepsilon^{\frac{1}{n_j}}  + o\left(\varepsilon^{\frac{1}{n_j}}\right)  \\  &   k \in \{1,2,...,r_j  \}, ~ l \in \{1,2,...,n_j \},  
\end{align*}
where $\{\zeta_{jk}\}^{r_j}_{k = 1}$ is the root of the equation 
$ \text{det} \left(\mathcal{Q}_{j} B \mathcal{P}_{j} - \zeta E_j \right) = 0 $.
\end{lemma}

\subsection*{The invariant subspaces of perturbed matrices}

The set $\mathcal{X} \subset \mathbb{R}^{n \times n}$ is an invariant subspace for the square matrix $A \in \mathbb{R}^{n \times n}$ if 
$$ A x  \in \mathcal{X} \qquad \forall x \in \mathcal{X}.   $$
Furthermore, we let a matrix $X \in \mathbb{R}^{n \times k}$, $k \in \{1,2,...,n \}$, whose columns form a basis for $\mathcal{X}$.  It is well established that there is a unique matrix $L \in \mathbb{R}^{k \times k}$ such that $A X = X L$.  The matrix $L$ is the representation of $A$ on $\mathcal{X}$ with respect to the basis $X$, and the eigenvalues of $L$ are eigenvalues of $A$.   

Now, we let the columns of $X$ form an orthonormal basis for 
$\mathcal{X}$.  Let the matrix $[X ~ Y]$ be orthonormal.  Consequently,  we obtain 
that 
\begin{align} \label{eqform}
[X ~ Y ]^\top A [X ~ Y] = \begin{bmatrix}   L &  X^\top A Y   \\ 0 & L_c    \end{bmatrix},  \quad L_c :=  Y^\top A Y.    
\end{align} 
Note that the set $\mathcal{X}$ is said to be a \textit{simple invariant subspace}  if the intersection between the eigenvalues of $L$ and $L_c$ is empty.  The following result is recalled from \cite[Theorem 1.5. Page 224]{Stewart90}.

\begin{lemma} [Spectral resolution] \label{Lemspecres}
Let $\mathcal{X}$ be a simple invariant subspace having the form \eqref{eqform} with respect to the matrix $(X,Y)$. Then, there exist matrices $X_c \in \mathbb{R}^{n \times (n-k)}$ and $Y_c \in \mathbb{R}^{n \times k}$ such that 
$$  [X ~ X_c]^{-} = [Y ~ Y_c]^\top   $$ 
and 
$  [X ~ X_c ]^- A [X ~ X_c] = \begin{bmatrix}   L &  0   \\ 0 & L_c    \end{bmatrix}.   $
\end{lemma}

Next,  we let $\mathcal{X}$ be a simple invariant subspace of $A$ and let a matrix $B \in \mathbb{R}^{n \times n}$.  We show that when $A-B$ is sufficiently small,  then there is a unique invariant subspace 
$\mathcal{Y}$ of $B$ that approaches $\mathcal{X}$ as $|A - B|$ goes to zero.   The following result is recalled from \cite[Theorem 15.5.1]{GOHLANROD}.

\begin{lemma} [Continuity of invariant subspaces] \label{lemcontin}
For a matrix $A \in \mathbb{R}^{n \times n}$ admitting an invariant subspace $\mathcal{X}$, the following properties are equivalent:
\begin{itemize}
\item The invariant subspace $\mathcal{X}$ is  simple. 

\item For each $\varepsilon>0$,  there exists $\delta>0$ such that,   each $B \in \mathbb{R}^{n \times n}$,  with $|A-B| \leq \delta$,   has a unique invariant subspace $\mathcal{Y}$ satisfying  $\text{dist}(\mathcal{X},\mathcal{Y}) \leq \varepsilon $, where $\text{dist}(\cdot,\cdot)$ is the Grassmann distance between linear subspaces \cite{ye2016schubert}.
\end{itemize}
\end{lemma}

Next,  we analyze the smooth variation of the invariant subspace $\mathcal{X}$ when the corresponding matrix $A$ depends smoothly on its parameters. The following lemma can be found in \cite[Theorem 18.7.2]{GOHLANROD}.    

\begin{lemma} [Analytic families of invariant subspaces] \label{lemanal}
Let $A : \mathcal{C} \rightarrow \mathcal{C}^{n \times n}$ be analytic. Given $z_o \in \mathcal{C}$,  we let $\Gamma$ be a contour in the complex plan such that the intersection between $\Gamma$ and the eigenvalues of $A(z_o)$ is empty.   Then, the invariant subspace $\mathcal{X}(z)$ of $A(z)$ corresponding to the eigenvalues of $A(z)$ lying inside $\Gamma$  is analytic on a neighborhood of $z_o$.    
\end{lemma}

\subsection*{The stable manifold theorem}

Consider the dynamical system 
\begin{align} \label{eqDS1}
\dot{x} = A x + g(x) \qquad x \in \mathbb{R}^n,
\end{align}
where $A$ has $n$ eigenvalues $\{\lambda_1(A),  \lambda_2(A), ... ,  \lambda_n(A) \}$ with 
\begin{align} \label{eqDS1-} 
\Re(\lambda_j(A))  \neq 0 \qquad \forall j\in \{1,2,...,n\}, 
\end{align}
i.e., the origin $x= 0$ is a hyperbolic equilibrium for \eqref{eqDS1}.   Moreover,  the function $g : \mathbb{R}^n \rightarrow \mathbb{R}^n$ is continuously differentiable,  and  there exists $\kappa \in \mathcal{K}$ 
such that
\begin{align} \label{eqDS2}
|g(x)| \leq \kappa(|x|) |x|.
\end{align}

Next,  we let $(E^s(A),E^u(A))$ be the spectral invariant subspaces associated with the stable and the unstable eigenvalues of $A$,  respectively.   Note that there exist positive constants $(c_s(A), c_u(A))$ (called, respectively,  forward and backward overshoots) such that
\begin{equation}
 \label{eqdecay}
\begin{aligned} 
 |e^{t A}  x_s| & \leq c_s(A) e^{-r_s(A) t}  |x_s| \qquad \forall t\geq 0, \quad \forall x_s \in E^s(A),   \\
 |e^{t A}  x_u| & \leq c_u(A) e^{r_u(A) t}  |x_u| \qquad \forall t \leq 0, \quad \forall x_u \in E^u(A),
\end{aligned}
\end{equation}
where 
\begin{equation}
 \label{eqdecay1}
\begin{aligned}
r_s(A) & := \{ \min |\Re(\lambda_j(A))| : \Re(\lambda_j(A)) < 0  \} 
\\
r_u(A) & := \{ \min |\Re(\lambda_j(A))| : \Re(\lambda_j(A)) > 0  \}.
\end{aligned}
\end{equation}

The following result and its proof can be found in \cite[Theorem~4.1]{LectureA}. 

\begin{lemma} [The stable manifold theorem] \label{lemstabmani}
Consider system \eqref{eqDS1} such that $g$ is continuously differentiable and \eqref{eqDS1-}-\eqref{eqDS2} hold.  
Then,  on the neighborhood of the origin $B_\gamma$, where $\gamma > 0$ is chosen so that, for some $\Delta > 0$, we have  
$$ c_s(A) \gamma + \kappa(\gamma) \left( \frac{c_s(A)}{r_s(A) - \mu(A)}  + \frac{c_u(A)}{r_u(A) + \mu(A)} \right) \Delta \leq \Delta,  $$
with $\mu(A) :=  \min \{r_u(A), r_s(A)\} / 2$,  there exists a continuously differentiable function $h^s : \text{Proj}_{E^s(A)}(B_\gamma) \rightarrow E^u(A)$ such that the stable manifold $W^s(0)$ is the graph of $h^s$ and the following properties hold:
\begin{enumerate}
\item $h^s(0) = 0$ and $\frac{\partial h^s}{\partial x_s}(0) = 0$,
\item $x(0) \in W^s(0)$ $\Longrightarrow$ $x(t) \in W^s(0)$ for all $t \geq 0$, 
\item when $x(0)  \in B_\gamma \backslash W^s(0)$ and $x(t)$ is bounded,  it follows that there exists $t_1>0$ such that 
$$ |x(t)| \geq \gamma/2  \qquad  \forall t \geq t_1.  $$
\end{enumerate} 
\end{lemma}

\section*{Appendix II: Auxiliary Lemmata} 

Given a function of two scalar variables that is smooth in one and only continuous in the other, the following original lemma shows the existence of a smooth approximation to any given 'nonuniform' degree of precision.   

\begin{lemma} \label{lemtech}
Consider a function $T : [0,1] \times [0,\alpha_o] \rightarrow \mathbb{R}^{n \times n}$ such that $\tau \mapsto T(\varepsilon, \tau)$ is continuous,  $\varepsilon \mapsto T(\varepsilon, \tau)$ is continuously differentiable, and $\tau \mapsto T(0,\tau)$  is continuously differentiable.   Then,   for each $\rho > 0$,  there exists $\hat{T} : [0,1] \times [0,\alpha_o] \rightarrow \mathbb{R}^{n \times n}$ continuously differentiable such that  
\begin{align}  
\hat{T} (0, \tau) & = T(0,\tau) &  \forall \tau \in [0,\alpha_o], \label{eqeigspa--}
\\
|\hat{T}(\varepsilon,\tau)  -  T(\varepsilon,\tau)  |_\infty & \leq \rho \varepsilon + o(\varepsilon)  & \forall \tau \in [0,\alpha_o],   \label{eqeigspa-}
 \end{align}
and 
\begin{align} \label{eqeigspa} 
\lim_{\varepsilon \rightarrow 0} \dot{T}(\varepsilon, \tau) = \lim_{\varepsilon \rightarrow 0} \frac{\partial \hat{T}(\varepsilon, \tau)}{\partial \tau}   =\dot{T}(0, \tau).
\end{align}
\end{lemma}
\begin{proof}
Since the matrix $T$ is continuously differentiable in $\varepsilon$ and continuous in $\tau$,  then it admits a first-order Taylor expansion of the form 
$$ 
T(\varepsilon, \tau) = T(0, \tau)  + a(\tau) \varepsilon + g(\varepsilon,\tau),  $$
where $a : [0,\alpha_o] \rightarrow \mathbb{R}^{n \times n}$ is continuous and $g :  [0,1] \times [0,\alpha_o) \rightarrow \mathbb{R}^{n \times n}$ enjoys the same continuity and smoothness properties as $T$.  Furthermore,   there exists $M>0$ such that, for each $\tau \in [0,\alpha_o]$,  we have  
$$   |g(\varepsilon, \tau)| \leq  M \varepsilon^2   \qquad  \forall  \varepsilon \in [0,1].   $$
Now,  we choose the matrix $\hat{T}$ as 
$$ 
\hat{T}(\varepsilon, \tau) = T(0, \tau)  + \hat{a}(\tau) \varepsilon,  $$
where $\hat{a} : [0,\alpha_o] \rightarrow \mathbb{R}^{n \times n}$ is a continuously differentiable approximation of $a$ on $[0,\alpha_o]$ satisfying 
$$ \sup_{\tau \in [0,\alpha_o]}  \{ | a(\tau) - \hat{a}(\tau) | \} \leq \rho.   $$
Note that to obtain the latter inequality,  we used Stone–Weierstrass theorem stating that every continuous function defined on a closed interval $[0, \alpha_o]$ can be uniformly approximated as closely as desired by a polynomial function  \cite{Stone1948}.

As a result,   \eqref{eqeigspa--} holds.  Furthermore,  we note that  
$$ T(\varepsilon, \tau) - \hat{T}(\varepsilon, \tau)  = (a(\tau) -  \hat{a}(\tau)) \varepsilon + g(\varepsilon, \tau), $$
which implies that  \eqref{eqeigspa-} also holds. 
Finally, \eqref{eqeigspa} holds under \eqref{eqeigspa--} and the continuous differentiability of $\hat{T}$. 
 \end{proof}

Consider the dynamical system 
\begin{align} \label{dynsys}
\dot x = f(x) \qquad x \in \mathbb{R}^n,
\end{align}
where $f: \mathbb{R}^n \rightarrow \mathbb{R}^n$ is continuously differentiable and the origin $x = 0$ is a hyperbolic equilibrium point.
The following lemma allows us to show that the propagation of the local stable manifold $W^s(0)$ using the backward solutions to \eqref{dynsys} is a null-measure set. 

\begin{lemma} \label{clcl1}
Consider the dynamical system \eqref{dynsys} and let $S_o \subset \mathbb{R}^n$ and $T>0$ such that, for each $x_o \in S_o$,  the solution $x(t)$ is well defined on $[0,T]$. 
Furthermore,  for each $t \in [0, T]$,  we  define the  reachable set 
$$  R^b(t,  S_o) := \left\{ y \in \mathbb{R}^n : y = x(t),~x(0) = x_o \in S_o \right\}.  $$ 
Then, if there exists $\tau \in[0,T]$ such that $\mu(R^b(\tau, S_o)) = 0$  then  $\mu(S_o) = 0$, where $\mu(\cdot)$ is the Lebesgue measure of $(\cdot)$. 
\end{lemma}
\begin{proof}
To find a contradiction,  we assume that $\mu(S_o) > 0$ and 
$\mu (R^b(\tau, S_o)) = 0$ for some $\tau \in [0, T]$.   Next,  we introduce the mapping $\phi_\tau  : S_o \rightarrow \mathbb{R}^n$ such that  $\phi_\tau(x_o) := x(\tau)$,  where $x(\tau)$ is the unique solution to \eqref{dynsys} starting from $x_o$.   
Using \cite[Theorem V.2.1]{hartman2002ordinary},   we conclude that  the mapping $\phi_\tau$  is continuous and clearly the reciprocal mapping satisfies $\phi_\tau^-(x_o) := x(-\tau, x_o)$.   Hence,  $\phi_\tau^-(\cdot)$ is also continuous and therefore $\phi_\tau$ is a homeomorphism; thus,  an open map.   Let us now fix $x_o \in S_o$ arbitrary such that there exists $U(x_o)$ an open set containing   $x_o$ that is contained in $S_o$,  the latter is possible to find since  $ \mu (S_o) \neq 0$.  Let $\phi_\tau(U(x_o))$ be the image of $U(x_o)$ by the homeomorphism $\phi_\tau$.   Since $\phi_\tau$ is a homeomorphism,  $\phi_\tau(U(x_o))$ is an open set containing  $\phi_{\tau}(x_o)$.  Hence,   $\mu(\phi_\tau(U(x_o))) \neq 0$.   However,  $\phi_\tau(U(x_o)) \subset R^b(\tau,  S_o)$ and we already assumed that $\mu (R^b(\tau,  S_o)) =0$, which yields to a contradiction.  
\end{proof}

\section*{Appendix III: Proofs} 

\subsection*{Proof of Lemma \ref{lemalgstab}}

By Assumption \ref{ass4}, $\gamma_o$ is globally orbitally asymptotically stable. Also, $\omega$ is globally attractive. Since $\gamma_o$ and $\gamma_1$ are disconnected, the solutions must converge to either one. Then, to show the statement of the Lemma, we  show that the solutions to \eqref{redsyst} converging to $\gamma_1$ must start from a null measure set which is the global stable manifold $W^s_o(0)$.  
Indeed,  to find a contradiction, we let $x_{mo} \notin W^s_o(0)$  such that the solution $x_m(t)$ to \eqref{redsyst} starting from $x_{mo}$ converges $\gamma_1$.  Now, using \lemstabmani,  we conclude the existence of $r > 0$ and $t_1 > 0$ such that 
\begin{align} \label{eqcontr}
 |x_m (t)| \geq r  \qquad  \forall t \geq t_1.   
\end{align}
As a result,  the solution $x_m(t)$ must converge to $\gamma_1 \backslash B_r$.  Now,  we let a strictly increasing 
sequence of times  $\{t_1, t_2, \cdots, \infty\}$ and the corresponding sequence of points 
$$ y_i := \text{Proj}_{\gamma_1 \backslash B_r}(x_m(t_i)) 
:= \text{argmin} \{  |y - x_m(t_i)| : y \in \gamma_1 \backslash B_r  \}.  $$
Note that,  since $x_m(t)$ converges to $\gamma_1 \backslash B_r$, it follows that  
$$ \lim_{i \rightarrow \infty} |y_i - x_m(t_i)| = 0.  $$ 
Now,  by definition of $\gamma_1$ and since it is compact, we conclude the existence of $T >0$ such that,  each solution $y_{mi}(t)$ to \eqref{redsyst} starting from $y_i$ satisfies 
$$  |y_{mi}(t)|  \leq  r/2  \qquad  \forall  t  \geq  T.   $$
 However, since the right-hand side in \eqref{redsyst} is continuously differentiable,  we the continuous dependence of the solutions on the initial data   \cite[Theorem V.2.1]{hartman2002ordinary} to conclude that,  for each $\epsilon >0$, there exists $i_\epsilon \in \{1,2,...\}$ such that, for each $i \geq i_\epsilon$, we have  
 $$  |y_{mi}(t) - x_m(t + t_i)  |  \leq  \epsilon  \qquad  \forall  t  \in [0, T].   $$
The latter implies that $x_m(t_i)$ must lie inside $B_r$,  
which contradicts \eqref{eqcontr}.

\subsection*{Proof of Lemma \ref{lem1}}
We start noting,  by definition of the  sets $\Omega$ and $\omega$,    that $(x_m, e_v) \in \Omega$ if and only if $e_v = 0$ and $x_m \in \omega$.  As a result,
$$  \left| (x_m, e_v) \right|_\Omega \leq \left| e_v \right| + \left| x_m \right|_\omega.  $$
Then,  to prove the lemma,  we show that

\begin{enumerate}[label={($\star$)},leftmargin=*]
\item \label{itemS00}   there  exists
$\varepsilon_1 \in \mathcal{K}$ such that,  for each $\rho >0$, for each $\varepsilon \leq \varepsilon_1(\rho)$,  and for each initial condition $(x_{mo}, e_{vo}) \in \mathbb{R}^n \times \mathbb{R}^{n(N-1)}$,  the solution $(x_m(t),e_v(t))$ to \eqref{singpertsyst}  satisfies 
\begin{align*} 
\lim_{t \rightarrow \infty} \left(\left| e_v(t) \right| + \left| x_m(t)\right|_{\omega} \right)  \leq \rho.   
\end{align*}  
\end{enumerate} 

In turn, to prove \ref{itemS00},  we first show that
\begin{enumerate}[label={($\star \star$)},leftmargin=*]
\item \label{itemS01}   there  exists $\varepsilon_v \in \mathcal{K}$ such that,  for each $\rho_v>0$,  for each $\varepsilon \leq \varepsilon_v(\rho_v)$, and for each initial condition $(x_{mo}, e_{vo}) \in \mathbb{R}^n \times  \mathbb{R}^{n(N-1)}$,  the solution $(x_m(t),e_v(t))$ to \eqref{singpertsyst} satisfies 
\begin{align} 
\lim_{t \rightarrow \infty} \left| e_v(t) \right|  \leq \rho_v.    \label{eqlim1}
\end{align}  
\end{enumerate} 

To prove \ref{itemS01},  we use the following four properties:

\begin{itemize}
\item  The dynamics of $e_v$ in \eqref{A11} can be expressed as in \eqref{A11biz};  namely, in the following form
\begin{equation*}
\begin{aligned}
\varepsilon \dot{e}_v &  =  - \left( \Lambda \otimes I_n \right) e_v + \varepsilon G_e(0,e_v) \\ & + \varepsilon \left[ G_e(x_m,e_v) -G_e(0,e_v) \right].
\end{aligned}
\end{equation*}

\item  By assumption,  system \eqref{singpertsyst} is globally 
(uniformly in $\sigma := 1/\varepsilon$) ultimately bounded.  That is,  
there exists $\varepsilon^\ast>0$ and  $r>0$ such that, 
for any $R \geq 0$, there exists $T_{R} \geq 0$  such that, 
for each $\varepsilon \leq  \varepsilon^\ast$ and  for each solution 
$(x_m(t),e_v(t))$ to \eqref{singpertsyst} starting from 
$(x_{mo},e_{vo}) \in \mathbb{R}^n \times \mathbb{R}^{n(N-1)}$,  
we have  
\begin{align*}
\left| ({x}_{mo},e_{vo}) \right| \leq R \Longrightarrow 
\left| (x_{m}(t),e_v(t) ) \right| \leq r  \quad \forall t \geq T_{R}.
\end{align*} 

\item We recall from \eqref{eq:born1}-\eqref{eq:born2}
the existence of a constant $d_r > 0$ such that, for each $(x_m,e_v) \in B_r$, we have 
\begin{equation}
\label{eqbornnew}
\begin{aligned}
 \left| G_e(0,e_v) \right|   &
 \leq d_r  \left| e_v \right|,
\\
\left|  G_e(x_m,e_v) - G_e(0,e_v) \right|  & \leq  d_r \left| x_m \right|. 
\end{aligned}
\end{equation}

\item Using Item (iv) in Lemma \ref{fact1},  we conclude that the system 
$$ \dot{e}_v = - \left( \Lambda \otimes I_n \right) e_v $$ 
is exponentially stable and there exists $P \in \mathbb{R}^{(N-1) \times (N-1)}$  symmetric and positive definite such that 
$$  P \Lambda  + \Lambda^\top P \leq - I_{N-1}.  $$
\end{itemize}

As a result,  combining the aforementioned four properties, we conclude that the time derivative of the Lyapunov function candidate    
\begin{align*}
  V_e(e_v) := e_v^\top (P \otimes I_n) e_v,
\end{align*}
under \eqref{eqbornnew} and along the solution $(x_m(t),e_v(t))$ to \eqref{singpertsyst} starting 
from $(x_{mo}, e_{vo}) \in B_R$,  satisfies 
\begin{align*}
\dot{V}_e(e_v(t)) & \leq  d_r  \lambda_{max}(P)  r \left| e_v(t) \right| 
 - \left[ \frac{1}{\varepsilon} - d_r \lambda_{max}(P) \right] \left| e_v(t) \right|^2 \\ &   \qquad \forall t \geq T_R.
\end{align*}
Now,  for $\varepsilon \leq \frac{1}{d_r \lambda_{\max}(P)}$,  we obtain 
\begin{align*}
\dot{V}_e(e_v(t)) & \leq  d_r  \lambda_{max}(P)  r 
 - \frac{\left[ 1 -  (1 + r) d_r \lambda_{max}(P)  \varepsilon  \right]}{\lambda_{max}(P) \varepsilon} V_e(e_v(t)) \\ &  \qquad \forall t \geq T_R.
\end{align*}
As a result,  
\begin{align*}
\lim_{t \rightarrow + \infty} |e_v(t)|^2 &  \leq   \lim_{t \rightarrow + \infty} \frac{V_e(e_v(t))}{\lambda_{min}(Q)} 
\\ & = 
\frac{d_r  \lambda_{max}(P)^2  r \varepsilon}{\left[ \lambda_{min}(Q)  - (1 + r) d_r \lambda_{max}(P) \lambda_{min}(Q)  \varepsilon  \right]}.  
\end{align*}  
Hence,  by introducing the class $\mathcal{K}$ function  
\begin{align*}
\varepsilon_v(\rho_v) & := \min \left\{ \varepsilon^\ast, 
 \frac{ 1 }{ (1 + r) d_r \lambda_{max}(P)  }  \right\}  \\ & \times 
 \frac{\rho_v}{ \frac{d_r  \lambda_{max}(P)^2  r}{(1 + r) d_r \lambda_{max}(P) \lambda_{min}(Q)} +   \rho_v},   
 \end{align*}
we conclude that, for each $\rho_v > 0$ and for each $\varepsilon \leq \varepsilon_v(\rho_v)$,  each solution $(x_m(t),e_v(t))$ to \eqref{singpertsyst} satisfies \eqref{eqlim1},  which proves \ref{itemS01}.

To complete the proof of \ref{itemS00},  we  rewrite the dynamics of $x_m$ in \eqref{A11} as in \eqref{A8}; namely, we have 
\begin{align*}
\dot{x}_m =  F_m (x_m) + G_m(x_m, e_v),
\end{align*}
where 
$$ G_m(x_m, e_v) := {F}_m(x_m,e_v) - F_m (x_m).  $$ 
Next, we recall from \eqref{ineqGm} the existence of $h: \mathbb{R}^n \times \mathbb{R}^{n(N-1)} \to \mathbb{R}_{\geq 0}$ continuous such that  
\begin{equation*} 
\big| G_m(x_m,e_v) \big|  \leq  h(x_m,e_v)|e_v|  \qquad  \forall (x_m,e_v) \in \mathbb{R}^n \times \mathbb{R}^{n(N-1)}.
\end{equation*} 
Hence,   for each $\rho_v > 0$,  for each $\varepsilon \leq 
\varepsilon_v(\rho_v)$,  and for each  solution $(x_m(t),e_v(t))$ to \eqref{singpertsyst} starting from $B_R$,  under the global ultimate boundedness of \eqref{singpertsyst} and \eqref{eqlim1},   we conclude the existence of $T_{R1} >T_R > 0$ and $c_r >0$ such that
$$ \left| G_m(x_m(t), e_v(t)) \right| \leq c_r \rho_v \qquad \forall t \geq T_{R1}.   $$  

Next,  we use Assumption \ref{ass4} to show that,  according to the notation in Appendix II.C,  the decomposition $\omega := \gamma_o \cup \gamma_1$ has no cycles  and we can only have 
$\gamma_o < \gamma_1$.   
Indeed,  since the periodic orbit $\gamma_o$ is asymptotically orbitally stable, we conclude that $R(\gamma_o) = \gamma_o$; hence,  
$ R(\gamma_o) \cap U(\gamma_1) =  \gamma_o \cap U(\gamma_1)$  
and, since $\gamma_o$ is invariant,  compact,  and $\gamma_o \cap 
\gamma_1 = \emptyset$,   it follows that 
$R(\gamma_o) \cap U(\gamma_1) = \emptyset$.  
Thus,  we cannot have  $\gamma_1 < \gamma_o$.  The latter also implies that we cannot have a 2-cycle. 
 Next,  we exclude the possibility of having a 
1-cycle  using asymptotic orbital stability of $\gamma_o$ and by definition of the set $\gamma_1$.   Finally,  the compact disconnected subset $\omega := \gamma_o \cup \{0\}$ is globally attractive for the unperturbed dynamics  \eqref{redsyst}.   Hence,  using \lemAE,  we conclude the existence of a continuously 
differentiable Lyapunov function $V : \mathbb{R}^n \rightarrow \mathbb{R}_{\geq 0}$,  class $\mathcal{K}^{\infty}$ functions $\alpha, \bar{\alpha}, \underline{\alpha}$, and a positive constant $c \geq 0$, such that
$$ \underline{\alpha}\left( \left|x_m\right|_{\omega} \right) \leq V(x_m) \leq \bar{\alpha} \left( \left|x_m\right|_{\omega}+c \right)  $$  
and 
$$  \frac{\partial V}{\partial x_m^\top}(x_m) F_m(x_m)  \leq - \alpha \left( \left|x_m\right|_{\omega} \right)  \qquad \forall x_m \in 
\mathbb{R}^n.   $$

Hence,  along the solution $(x_m(t),e_v(t))$ to \eqref{singpertsyst} starting from $B_R$,  the time-derivative of the Lyapunov function $V$ satisfies the following upper-bound
\begin{align*}
\dot{V}(x_m(t)) &  \leq - \alpha \left( \left|x_m(t)\right|_{\omega} \right) +  \left| \frac{\partial V}{\partial x_m^\top}(x_m(t)) \right|  c_r \rho_v 
\\ & \forall t \geq T_{R1}.   
\end{align*}
Next,  using the global ultimate boundedness of \eqref{singpertsyst} and the continuity of $\frac{\partial V}{\partial x_m^\top}$, we conclude the existence of $b_r$ and $T_{R2} > T_R > 0$ such that
$$ \left| \frac{\partial V}{ \partial x_m^\top}(x_m(t)) \right| \leq b_r \qquad \forall t \geq T_{R2}.  $$ 
As a result,  for all $t \geq \sup \left\{T_{R1}, T_{R2} \right\}$, we have 
$$ \dot{V}(x_m(t)) \leq - \alpha \left( \underline{\alpha}^{-} \left( V(x_m(t)) \right) \right) + b_r c_r  \rho_v.  $$
The last step invokes the comparison lemma \cite[Lemma 2.5]{KHALIL96}, that establishes the existence of $\mu \in \mathcal{K}^{\infty}$  such that 
$$ \lim_{t \rightarrow \infty} V(x_m(t)) \leq \mu \left(b_r c_r  \rho_v \right),  $$ 
and consequently, 
$$ \lim_{t \rightarrow \infty} \left|x_m(t)\right|_{\omega} = \lim_{t \rightarrow \infty} \left|(x_m(t), 0)\right|_{\Omega} \leq \underline{\alpha}^-\left(\mu \left(b_r c_r  \rho_v \right)\right).$$
Finally,  to complete the proof,  it is enough to take  $ \varepsilon_1(\rho) :=   \varepsilon_v (\chi^{-1} (\rho))$, where 
$$   \chi(\rho_v) :=   \underline{\alpha}^-\left(\mu \left(b_r c_r  \rho_v \right)\right) + \rho_v. $$

\subsection*{Proof of Lemma \ref{lem2}}
The proof is based on a direct application of Lemma \ref{thmAnosov}.   The first item in Lemma \ref{thmAnosov} is holds under Assumption \ref{ass1}, 
the second item holds with $e_v = h(x_m) = 0$,  
the third item is satisfied since the boundary-layer model of \eqref{singpertsyst} is the linear system $\dot{e}_v  = - ( \Lambda \otimes I_n ) e_v$ with $\Lambda$ Hurwitz since the graph is connected; see Lemma \ref{fact1},
and the last item in Lemma \ref{thmAnosov} holds under Assumption \ref{ass4}.
\hfill $\blacksquare$

\subsection*{Proof of Lemma \ref{lem3}} 



Consider a periodic orbit $\Gamma_\varepsilon \subset T_\rho$,  for some $\rho \in (0,1]$,  generated by an  $\alpha_{\varepsilon}$-periodic solution to \eqref{singpertsyst} denoted by $\bar{x}_{\varepsilon}(t) := (x_{m\varepsilon}(t), e_{v\varepsilon}(t))$. 
Also,  we let the $\alpha_o$-periodic solution 
to \eqref{redsyst},  denoted by $x_{mo}(t)$,  generating the periodic orbit $\gamma_o$ introduced in  Assumption \ref{ass4}.  We also let $\tilde{\alpha} := \alpha_{\varepsilon} -$ $\alpha_o$.  Finally,  we introduce the following error coordinate 
$ \tilde{x} =\bar{x} - \bar{x}_{\varepsilon} $.

In the coordinates $\tilde{x}$,  we re-express system \eqref{singpertsyst} as follows
\begin{align} \label{eq1:lem3}
\dot{\tilde{x}} = A(\bar{x}_\varepsilon(t)) \tilde{x}  + 
g(\bar{x}_\varepsilon(t), \tilde{x}),
\end{align}
where 
\begin{align*}
 A(\bar{x}_\varepsilon(t))  :=  & 
\begin{bmatrix}  
\sfrac{\partial F_m}{\partial x_m^\top} &   \sfrac{\partial G_m}{\partial e_v^\top}     
\\[12pt]
\sfrac{\partial G_e}{\partial x_m^\top} &  - \sfrac{1}{\varepsilon} (\Lambda \otimes I_n)  + \sfrac{\partial G_e }{\partial e_v^\top}
\end{bmatrix}\vline_{
  \begin{minipage}{20mm}
   \ \\[5mm]   \footnotesize  $\begin{array}{l}
    x_m  =  x_{m\varepsilon}(t)\\
    e_v  =  e_{v\varepsilon}(t),  
  \end{array}$
  \end{minipage}
}
\\[2pt]
%
 g( \bar{x}_\varepsilon(t),  \tilde{x} )  := & \, \bar{F}(\bar{x}_\varepsilon(t) + \tilde{x}) - \bar{F}(\bar{x}_\varepsilon(t)) -   A(\bar{x}_\varepsilon(t)) \tilde{x}.  
 \end{align*}
Note that $g( \bar{x}_\varepsilon(t),  \tilde{x} )$ is continuously differentiable in $\tilde{x}$,  continuous in $\bar{x}_\varepsilon(t)$,  and we can find  $\kappa \in \mathcal{K}$ (independent of $\varepsilon$ since
$\Gamma_\varepsilon \subset T_{\rho} \subset T_{\rho = 1}$) such that 
\begin{align} \label{eqboundNL}
| g( \bar{x}_\varepsilon(t),  \tilde{x} ) |  \leq  \kappa(\tilde{x}) |\tilde{x}|.   
\end{align}  
To prove the lemma,  we follow the following steps:

\begin{enumerate}
\item We start re-describing  \eqref{eq1:lem3} using the new time scale $\tau := \frac{\alpha_o}{\alpha_\varepsilon} t$, which gives us 
\begin{equation*}
\begin{aligned} 
\tilde{x}' & := \frac{d \tilde{x}}{d \tau} = \frac{\alpha_\varepsilon}{\alpha_o} A \left(\bar{x}_\varepsilon \left(\frac{\alpha_\varepsilon}{\alpha_o} \tau\right) \right) \tilde{x} 
\\ &  := A_1 \left(\bar{x}_1(\tau)\right) \tilde{x} + \left(\frac{\alpha_\varepsilon}{\alpha_o} \right)
 g\left( \bar{x}_1 (\tau),   \tilde{x}_\varepsilon \right), 
\end{aligned}
\end{equation*}
where 
$$ A_{1}(\cdot) := \frac{\alpha_\varepsilon}{\alpha_o} A(\cdot)  \quad  \text{and} \quad  \bar{x}_{1}(\tau) := \bar{x}_{\varepsilon} \left(\frac{\alpha_\varepsilon}{\alpha_o} \tau\right).  $$

\item  Next,  we re-express  $A_1(\bar{x}_1(\tau))$  as 
\begin{equation}
\label{eqsysLin1}
\begin{aligned} 
A_1(\bar{x}_1(\tau)) = A \left(\bar{x}_o(\tau)\right)  +      \Delta_{1}(\bar{x}_1(\tau),  \bar{x}_o(\tau),  \varepsilon,  \rho, \tilde{\alpha}), 
\end{aligned}
\end{equation}
where  $ \bar{x}_o(\tau) := (x_{mo}(\tau),0)$ and 
\begin{align*}
& \Delta_{1}(\bar{x}_1(\tau),  \bar{x}_o(\tau),  \varepsilon,  \rho) := \\ & A_{1}(\bar{x}_1(\tau)) - 
A_{1} (\bar{x}_o(\tau)) + \left( \frac{\tilde{\alpha}}{\alpha_o}  \right) A(\bar{x}_o(\tau)). 
\end{align*}
Furthermore,  we show that 
\begin{equation}
\label{eqDelta}
\begin{aligned} 
\lim_{(\rho,\varepsilon,\tilde{\alpha}) \rightarrow (0,0,0)} & \left| \Delta_{1}(\bar{x}_1(\tau),\bar{x}_o(\tau), \varepsilon, \rho,\tilde{\alpha}) \right|  = 0 \\ &   \forall \tau \in [0,\alpha_o]
\end{aligned}
\end{equation} 
by following two steps.

\begin{itemize}
\item  We first apply Tikhonov Theorem (see \thmtykho\ and \RemThyk) on the singularly perturbed system \eqref{singpertsyst},  to conclude the existence of 
$\varepsilon^\ast>0$ and $M>0$ such that, for each $\varepsilon \in [0, \varepsilon^\ast]$,  the solution  
$\bar{x}_{\varepsilon}(t)$ to \eqref{singpertsyst} and the signal $\bar{x}_o(t)$,  with $x_{mo}(t)$ solution to \eqref{redsyst},   satisfy
\begin{equation}
\label{eqtheyko} 
\begin{aligned} 
| \bar{x}_{\varepsilon}(0) - \bar{x}_{o}(0) | \leq \rho  &
\Longrightarrow
| \bar{x}_{\varepsilon}(t) - \bar{x}_{o}(t) | \leq M \rho \\ & \qquad \forall t \in [0, \alpha_\varepsilon].
\end{aligned}
\end{equation}

\item Now,  using \eqref{eqtheyko},  we conclude that, for each $\tau \in [0, \alpha_o]$,  we have 
\begin{equation}
\label{chgtdepreuve1}
\begin{aligned} 
\bar{x}_\varepsilon \left( \frac{\alpha_\varepsilon}{\alpha_o}\tau \right) &  =  \bar{x}_o \left( \frac{\alpha_\varepsilon}{\alpha_o}\tau \right) + O(\rho)  \nonumber \\ & 
= \bar{x}_{o}(\tau) - \Delta_{\bar{x}_o} \left( \tau, \varepsilon, \tilde{\alpha} \right) + O(\rho),
\end{aligned}
\end{equation}
where 
$$ \Delta_{\bar{x}_{o}} \left( \tau, \varepsilon, \tilde{\alpha} \right) :=   \bar{x}_o \left( \frac{\alpha_\varepsilon}{\alpha_o}  \tau  - \left(  \frac{\tilde{\alpha}}{\alpha_o}  \right) \tau \right) - \bar{x}_o \left( \frac{\alpha_\varepsilon}{\alpha_o} 
\tau \right).  $$

Under Item (i) in Lemma \ref{fact1},  there exists $M_\rho >0$ such that,  for each $\tau \in [0,\alpha_o]$,   we have 
\begin{equation*}
\begin{aligned} 
\left| \Delta_{\bar{x}_o} \left( \tau, \varepsilon, \tilde{\alpha} \right) \right| 
\leq & \sup_{0 \leq t \leq \alpha_\varepsilon}  
\left|  \dot{\bar{x}}_o \left( t \right) \right| \left| \tilde{\alpha} \right| 
\\ & \leq  \sup_{0 \leq t \leq \alpha_\varepsilon} \left| \bar{F} (\bar{x}_{o}(t)) \right|  \left| \tilde{\alpha}  \right|  \\ &  \leq  M_\rho   \left| \tilde{\alpha} \right|,  
\end{aligned}
\end{equation*}
where, for $\bar{x}_o(t) := (x_m(t),0)$, we have 
$$ \bar{F} (\bar{x}_o(t)) :=   
\begin{bmatrix}
 {F}_m(x_m(t)) + G_m(x_m(t),0)   
 \\
 G_e(x_m(t),0)
 \end{bmatrix}.  $$
As a result,  since both $\bar{x}_{1}(\tau)$ and $\bar{x}_{o}(\tau)$ are $\alpha_o$-periodic,  we conclude that
\begin{align} \label{chgtdepreuve8}
 \lim_{(\rho,\varepsilon,\tilde{\alpha}) \rightarrow (0,0,0)} | \bar{x}_{1}(\tau) - \bar{x}_{o}(\tau) | = 0 \qquad \forall \tau \geq 0.
\end{align}
\end{itemize}
As a result,  using \eqref{chgtdepreuve8} and Item (i) in Lemma \ref{fact1},  we conclude that \eqref{eqDelta} holds. 

\item  Now,  we decompose the matrix $A(\bar{x}_o(\tau))$ as 
\begin{align*}
A(\bar{x}_o(\tau)) & := 
\begin{bmatrix}  
 \sfrac{\partial F_m}{\partial x_m^\top}(x_{mo}(\tau)) 
&   0
\\[4pt]
0 &  0
\end{bmatrix} 
 + \frac{1}{\varepsilon} 
A_{\varepsilon}(\bar{x}_o(\tau)),
\end{align*}
where 
\begin{align*}
A_{\varepsilon} 
 & :=  
\begin{bmatrix}  
0  &  0   
\\[4pt]
0 & - \Lambda \otimes I_n
\end{bmatrix} 
+ \varepsilon
\begin{bmatrix}  
0  &   \sfrac{\partial G_m}{\partial e_v^\top}     
\\[9pt]
\sfrac{\partial G_e}{\partial x_m^\top} &  \sfrac{\partial G_e }{\partial e_v^\top}
\end{bmatrix}\vline_{
  \begin{minipage}{20mm}
   \ \\[1mm]   \footnotesize  $\begin{array}{l}
    x_m  =  x_{mo}(\tau)\\
    e_v  =  0.
  \end{array}$
  \end{minipage}
}  
\end{align*}
Under \lemregpert,  we know that  $A_{\varepsilon}(\bar{x}_o(\tau))$ admits $n$ eigenvalues of the form 
\begin{align*} 
\lambda_{j}(A_{\varepsilon}(\bar{x}_o(\tau))) & =    o(\varepsilon) \qquad \forall j = \{1,2,...,n\}.
\end{align*}
Furthermore,  we consider the non-singular matrix $T(\varepsilon,\tau)$ transforming $A_\varepsilon(\bar{x}_o(\tau))$, as in \Lemspecres, into the  block-diagonal form
 \begin{align*}
J_{\varepsilon}(\tau) :=  
\begin{bmatrix} 
\Lambda_o(A_\varepsilon(\bar{x}_o(\tau))) & 0_{n \times n(N-1)}  
\\ 
 0_{n(N-1) \times n}  & \Lambda_1(A_\varepsilon(\bar{x}_o(\tau))) \end{bmatrix},
\end{align*}
 where $ \Lambda_o(A_\varepsilon(\bar{x}_o(\tau))) \in \mathbb{R}^{n \times n} $  is the representation of $A_\varepsilon(\bar{x}_o(\tau))$ on the invariant spectral subspace corresponding to $ \{  \lambda_{j}(A_{\varepsilon}(\bar{x}_o(\tau))) \}^n_{j = 1}$.  Similarly,  
 $\Lambda_1(A_\varepsilon(\bar{x}_o(\tau))) \in \mathbb{R}^{n(N-1) \times n(N-1)} $  is the representation of $A_\varepsilon(\bar{x}_o(\tau))$ on the invariant spectral subspace corresponding to the remaining eigenvalues. 
 
Now,  for $\varepsilon > 0$ sufficiently small, we use \lemcontin\ to conclude that the transformation matrix $T(\varepsilon, \tau)$ is continuous.  Furthermore, using \lemanal,  we conclude that $T(\varepsilon, \tau)$ is analytic (smooth) in $\varepsilon$.  As a result,  for each $\tau \in [0,\alpha_o)$,     we have $\lim_{\varepsilon \rightarrow 0} T(\varepsilon, \tau)  = I_{nN}$ and 
\begin{equation} 
\label{eqmat}
 \begin{aligned} 
 \lim_{\varepsilon \rightarrow 0} 
 \frac{1}{\varepsilon}  \Lambda_o(A_\varepsilon(\bar{x}_o(\tau)))  = 0,  \\
 \lim_{\varepsilon \rightarrow 0}  \left[ \Lambda_1(A_\varepsilon(\bar{x}_o(\tau)))  +  (\Lambda \otimes I_n) \right]  = 0.  
 \end{aligned}
 \end{equation}
 
\item  At this point,  we introduce the change of coordinates  
$\tilde{x}_\varepsilon :=	\hat{T} (\varepsilon, \tau) \tilde{x}$,
where $\hat{T}$ is a continuously  differentiable non-singular approximation of $T$ chosen according to Lemma \ref{lemtech}.   In the new coordinates,  \eqref{eqsysLin1} becomes 
\begin{align*} 
 \tilde{x}'_\varepsilon & =  \hat{T} \begin{bmatrix}  
  \sfrac{\partial F_m}{\partial x_m^\top}(x_{mo} (\tau)) &   0
\\[4pt]
0 &  0
\end{bmatrix}   
{\hat{T}}^- \tilde{x}_\varepsilon   + \hat{T} A_{\varepsilon}(\bar{x}_o(\tau)) {\hat{T}}^-  \tilde{x}_\varepsilon 
\\ & 
+  \left[ \hat{T}  \Delta_1 {\hat{T}}^-  + \dot{\hat{T}} {\hat{T}}^- \right]  \tilde{x}_\varepsilon  +
\left(\frac{\alpha_\varepsilon}{\alpha_o} \right)
\hat{T} g\left( \bar{x}_1 (\tau),  \hat{T}^- \tilde{x}_\varepsilon \right). 
\end{align*}
The latter can be further expressed as 
\begin{align*} 
\tilde{x}'_\varepsilon & =    \begin{bmatrix}  
  \sfrac{\partial F_m}{\partial x_m^\top}(x_{mo} (\tau)) &   0
\\[4pt]
0 &  0
\end{bmatrix}   
 \tilde{x}_\varepsilon  +  \frac{J_{\varepsilon}(\tau)}{\varepsilon}  \tilde{x}_\varepsilon  
\\ & 
 + [T - I_{nN} ]  \begin{bmatrix}  
 \sfrac{\partial F_m}{\partial x_m^\top}(x_{mo} (\tau)) &   0
\\[4pt]
0 &  0
\end{bmatrix}   
[T^-  - I_{nN}]  \tilde{x}_\varepsilon
\\ & +  \left[ \hat{T}  \Delta_1 {\hat{T}}^-  + \dot{\hat{T}} {\hat{T}}^- \right]  \tilde{x}_\varepsilon  
\\ & +  [\hat{T} - T ]  A(\bar{x}_o(\tau)) [\hat{T}^- - T^- ]  \tilde{x}_\varepsilon 
\\ &
+ \left(\frac{\alpha_\varepsilon}{\alpha_o} \right)
\hat{T} g\left( \bar{x}_1 (\tau),  \hat{T}^- \tilde{x}_\varepsilon \right),
\end{align*}
which allows us to write
\begin{align*} 
 \tilde{x}'_\varepsilon &
= 
 \begin{bmatrix}  
\sfrac{\partial F_m}{\partial x_m^\top}(x_{mo} (\tau)) &   0
\\[4pt]
0 &  \frac{1}{\varepsilon} \Lambda_1(A_\varepsilon(\bar{x}_o(\tau)))
\end{bmatrix}  \tilde{x}_\varepsilon
\\ & +
 \begin{bmatrix}  
 \frac{1}{\varepsilon} \Lambda_o(A_\varepsilon(\bar{x}_o(\tau)))  &   0
\\[4pt]
0 & 0
\end{bmatrix}  \tilde{x}_\varepsilon
 \\ & + [T - I_{nN} ]  \begin{bmatrix}  
\sfrac{\partial F_m}{\partial x_m^\top}(x_{mo} (\tau)) &   0
\\[4pt]
0 &  0
\end{bmatrix}   
[T^-  - I_{nN}]  \tilde{x}_\varepsilon
 \\ &
+  \left[ \hat{T}  \Delta_1 {\hat{T}}^-  + \dot{\hat{T}} {\hat{T}}^- \right]  \tilde{x}_\varepsilon  
\\ &
+  [\hat{T} - T ]  A(\bar{x}_o(\tau)) [\hat{T}^- - T^- ]  \tilde{x}_\varepsilon 
\\ &
+ \left(\frac{\alpha_\varepsilon}{\alpha_o} \right)
\hat{T} g\left( \bar{x}_1 (\tau),  \hat{T}^- \tilde{x}_\varepsilon \right).
\end{align*}

Finally,  we let 
\begin{align*}
& A_2(\tau,\varepsilon)  := \\ &
\begin{bmatrix}
\sfrac{\partial F_m}{\partial x_m^\top}(x_{mo} (\tau)) &   0 \\
0 &  \frac{ (\Lambda \otimes I_n)  -  \left[ \Lambda_1(A_\varepsilon(\bar{x}_o(\tau))) + (\Lambda \otimes I_n) \right] }{ - \varepsilon}
\end{bmatrix},
\end{align*}
\begin{align*}
& A_3(\tau, \rho, \varepsilon,\tilde{\alpha})  :=  
\begin{bmatrix}  
 \sfrac{\Lambda_o(A_\varepsilon(\bar{x}_o(\tau)))}{\varepsilon}  &   0
\\[4pt] 
0 &   0
\end{bmatrix}  
  \\ & + 
 [T - I_{nN} ]  \begin{bmatrix}  
\sfrac{\partial F_m}{\partial x_m^\top}(x_{mo} (\tau)) &   0 
\\
0 &  0
\end{bmatrix}   
[T^-  - I_{nN}] 
 \\ &
+  \left[ \hat{T}  \Delta_1 {\hat{T}}^-  + \dot{\hat{T}} {\hat{T}}^- \right]    +  [\hat{T} - T ]  A(\bar{x}_o(\tau)) [\hat{T}^- - T^- ],
\end{align*}
\begin{align*}
h_\varepsilon(\tau, \tilde{x}_\varepsilon)   := \left(\frac{\alpha_\varepsilon}{\alpha_o} \right)
\hat{T} g\left( \bar{x}_1 (\tau),  \hat{T}^- \tilde{x}_\varepsilon \right).
 \end{align*}
 
\item  As a last step,   applying Floquet Theory,   we conclude the existence of  a non-singular 
$\alpha_o$-periodic matrix $P_o : \mathbb{R}_{\geq 0} \rightarrow \mathbb{R}^{n \times n}$ and $B_o \in \mathbb{R}^{n \times n}$,  such that 
$$  \dot{P}_o(\tau) +  P_o(\tau)^- \frac{\partial F_m(x_{mo} (\tau))}{\partial x_m^\top} P_o(\tau) = B_o.   
$$

Now, using the change of coordinates $y := \mathcal{P}(\tau)^- \tilde{x}_\varepsilon$,  with  
$$ \mathcal{P} (\tau) :=  \text{blkdiag} \left\{ P_o(\tau), I_{n(N-1)}  \right\},  $$ 
 we obtain 
\begin{align} \label{systrans}
y' = B(\varepsilon,  \tau)  y  +  \bar{A}_3(\tau,\rho, \varepsilon, \tilde{\alpha})  y + 
\mathcal{P}(\tau)^-  h_\varepsilon \left(   \tau ,   \mathcal{P}(\tau) y \right),  
\end{align}
where 
\begin{align*}
 \bar{A}_3(\tau,\rho, \varepsilon, \tilde{\alpha}) & :=   \mathcal{P}(\tau)^- A_3(\tau,\rho, \varepsilon, \tilde{\alpha}) \mathcal{P}(\tau),
\\ 
  B(\varepsilon, \tau) & :=  \begin{bmatrix} B_o & 0 \\ 0 &  \frac{  ( \Lambda \otimes I_n)  - \left[ \Lambda_1(A_\varepsilon(\bar{x}_o(\tau))) + (\Lambda \otimes I_n) \right] }{- \varepsilon} \end{bmatrix}. 
 \end{align*} 
\end{enumerate}

 Note that $\bar{A}_3$ is continuous in its arguments and,  using  \eqref{eqDelta},  \eqref{eqmat},  and Lemma \ref{lemtech},  we conclude that 
$$ \lim_{(\rho,\varepsilon,\tilde{\alpha}) \rightarrow 0} |\bar{A}_3(\tau, \varepsilon, \rho, \tilde{\alpha})| = 0 \qquad  \forall \tau \in [0,\alpha_o].   $$
 Furthermore,  for the block-diagonal matrix $B(\varepsilon, \tau)$, we use Assumption \ref{ass4} and Lemma \ref{lemfloq},  to conclude that the upper block  $B_o$ is Hurwitz.  Moreover,  the lower block $ \frac{  ( \Lambda \otimes I_n)  - \left[ \Lambda_1(A_\varepsilon(\bar{x}_o(\tau))) + (\Lambda \otimes I_n) \right] }{- \varepsilon}$,  for sufficiently small $\varepsilon>0$,  has all its characteristic multipliers  inside the unit circle since   $\Lambda$ is Hurwitz and $$ \lim_{\varepsilon \rightarrow 0} \left[ \Lambda_1(A_\varepsilon(\bar{x}_o(\tau))) + (\Lambda \otimes I_n) \right] = 0 \qquad \forall \tau \in [0,\alpha_o].  $$ 
Finally,  $h_\varepsilon(\tau,   \tilde{x}_\varepsilon)$ is continuous in $\tau$ and continuously differentiable in $ \tilde{x}_\varepsilon$, and under \eqref{eqboundNL},  we can find $\kappa_h \in \mathcal{K}$  such that $$  \mathcal{P}(\tau)^-  h_\varepsilon \left(   \tau ,   \mathcal{P}(\tau) y \right)  \leq \kappa_h(|y|) |y|  \qquad \forall \tau \in [0, \alpha_o].   $$ 
 
Hence, we can find $\rho^{**}>0$ and $\varepsilon^{**}>0$ such that, for each $\varepsilon  \in (0,\varepsilon^{**}]$, the origin $y = 0$ for \eqref{systrans}  is uniformly exponentially stable on the set $\{ y \in \mathbb{R}^{nN} : |y| \leq  \rho^{**}  \} $.

\subsection*{Proof of Lemma \ref{lem6}}
To establish the proof,  we start analyzing the linearization of \eqref{singpertsyst} around the origin,  which is given by  
$$ \dot{\bar{x}} =  (A_\varepsilon/\varepsilon)  \bar{x},  $$
where 
$$ A_\varepsilon :=  
\left\{
\begin{bmatrix}  
0&  0
\\ & \\
0 &  - (\Lambda \otimes I_n)  
\end{bmatrix}  
+ \varepsilon
\begin{bmatrix}  
 \sfrac{\partial F_m}{\partial x_m^\top}(0) &   \sfrac{\partial G_m }{\partial e_v^\top}(0,0)     \\ & \\
 \sfrac{\partial G_e}{\partial x_m^\top}(0) &  \sfrac{\partial G_e}{\partial e_v^\top}(0,0)     
\end{bmatrix}   \right\}.  $$ 

Using \lemregpert,  we know that  $A_{\varepsilon}$ admits $n$ eigenvalues of the form 
\begin{align*} 
\lambda_{j}(A_{\varepsilon}) & =   \lambda_j\left(  \frac{\partial F_m}{\partial x_m^\top}(0)  \right) \varepsilon +  o(\varepsilon) \qquad \forall j = \{1,2,...,n\}.
\end{align*}
Furthermore,  we consider the non-singular matrix 
$T(\varepsilon)$ transforming  $A_\varepsilon$,  as in \Lemspecres,  into the  block-diagonal form
 \begin{align*}
J_{\varepsilon} :=  
\begin{bmatrix} 
\begin{bmatrix} 
\Lambda_u(A_\varepsilon)  & 0
\\
0 &  \Lambda_s(A_\varepsilon) 
\end{bmatrix}
& 0_{n \times n(N-1)}  
\\ 
0_{n(N-1) \times n}  & \Lambda_1(A_\varepsilon) 
\end{bmatrix},
\end{align*}
 where $ \Lambda_u(A_\varepsilon) \in \mathbb{R}^{k \times k}$  
 is the representation of $A_\varepsilon$ on the spectral subspace corresponding to the unstable eigenvalues in 
 $\{\lambda_{j}(A_{\varepsilon})\}^n_{j = 1}$,  $ \Lambda_s(A_\varepsilon) \in \mathbb{R}^{(n-k) \times (n-k)}$  
 is the representation of $A_\varepsilon$ on the spectral subspace corresponding to the stable eigenvalues in 
 $\{\lambda_{j}(A_{\varepsilon})\}^n_{j = 1}$.  Similarly,  $\Lambda_1(A_\varepsilon) \in \mathbb{R}^{n(N-1) \times n(N-1)}$  is the representation of $A_\varepsilon$ on the spectral subspace corresponding to the remaining eigenvalues of $A_\varepsilon$.   
 
 For $\varepsilon > 0$ sufficiently small,  we use  \lemanal\ to conclude that $T(\varepsilon)$ is analytic (smooth) in $\varepsilon$.  As a result,  we have 
\begin{equation} 
\label{equseful}
 \begin{aligned} 
\Lambda_u(A_\varepsilon) &  = \varepsilon
\left[   \frac{\partial F_m}{\partial x_m^\top}(0) \right]_u  + o(\varepsilon), 
\\
\Lambda_s(A_\varepsilon) &  = \varepsilon
\left[   \frac{\partial F_m}{\partial x_m^\top}(0) \right]_s + o(\varepsilon), 
\\
  \Lambda_1(A_\varepsilon)  &  = -(\Lambda \otimes I_n) + O(\varepsilon).    
 \end{aligned}
 \end{equation}
 where $\left[  \frac{\partial F_m}{\partial x_m^\top}(0) \right]_u \in \mathbb{R}^{k \times k}$ is a representation $\frac{\partial F_m}{\partial x_m^\top}(0)$ on its unstable subspace and  
 $\left[ \frac{\partial F_m}{\partial x_m^\top}(0) \right]_s \in  \mathbb{R}^{(n-k) \times (n-k)}$ is a representation $ \frac{\partial F_m}{\partial x_m^\top}(0)$ on its stable subspace.
 
 According to the aforementioned  properties,  we conclude the existence of  $\varepsilon^{\star} > 0$ such that,  for each 
 $\varepsilon \in (0,\varepsilon^{\star}]$,  the origin is a hyperbolic equilibrium  for \eqref{singpertsyst}.   On the other hand,  following the notation of \lemstabmani,  we conclude that
\begin{equation}
\label{eqImpprop}
\begin{aligned}
r_s(A_\varepsilon/\varepsilon) & := \{ \min |\Re(\lambda_j(A_\varepsilon/\varepsilon))| : \Re(\lambda_j(A_\varepsilon/\varepsilon)) < 0  \} 
\\ & = r_s \left(  \frac{\partial F_m}{\partial x_m^\top}(0) \right) + O(\varepsilon), 
\\
r_u(A_\varepsilon/\varepsilon) & := \{ \min |\Re(\lambda_j(A_\varepsilon/\varepsilon))| : \Re(\lambda_j(A_\varepsilon/\varepsilon)) > 0  \}
\\ &
= r_u\left( \frac{\partial F_m}{\partial x_m^\top}(0) \right) + O(\varepsilon).
\end{aligned}
\end{equation}
Moreover,  the forward and the backward overshoots of $A_\varepsilon/\varepsilon$ satisfy 
\begin{equation}
\label{eqImpprop1}
\begin{aligned} 
c_s(A_\varepsilon/\varepsilon) & = c_s(A_\varepsilon) = c_s(J_\varepsilon) = c_s \left( \text{blkdiag} \{ \Lambda_s(A_\varepsilon), \Lambda_1 (A_\varepsilon) \}  \right),
\\
c_u(A_\varepsilon/\varepsilon) & = c_u(A_\varepsilon) = c_u(J_\varepsilon) = c_u(\Lambda_u(A_\varepsilon)).
\end{aligned}
\end{equation}
Using \eqref{equseful}, we conclude that for $\varepsilon^{*}$ small enough,  the forward and backward overshoots $(c_u(\Lambda_u(A_\varepsilon)),  c_s(\Lambda_s(A_\varepsilon))  )$ can be chosen as 
\begin{equation*}
\begin{aligned} 
 c_s(\Lambda_s(A_\varepsilon))  & = c_s \left( \left[ \frac{\partial F_m}{\partial x_m^\top}(0)  \right]_s \right) + c_s \left( \Lambda  \right),
\\
 c_u(\Lambda_u(A_\varepsilon)) & = c_u  \left(  \left[ \frac{\partial F_m}{\partial x_m^\top}(0)  \right]_u  \right)  \qquad \qquad  \qquad \forall \varepsilon \in (0,\varepsilon^*].
\end{aligned}
\end{equation*}

Now,  if we let
$$ g(\bar{x}) := \bar{F} (\bar{x})  - \left(A_\varepsilon/\varepsilon\right) \bar{x},  $$ 
we conclude that $g$ does not depend on $\varepsilon$ and we can find $\kappa \in \mathcal{K}$ such that 
\begin{align} \label{eqDS2bis}
|g(\bar{x})| \leq \kappa(|\bar{x}|) |\bar{x}|.
\end{align}

As a result,  under \eqref{eqImpprop}, \eqref{eqImpprop1}, and \eqref{eqDS2bis},  we conclude that,  for $\varepsilon^\star > 0$ sufficiently small,  we can find $\gamma > 0$ and $\Delta > 0$ such that,  for each $\varepsilon \in (0, \varepsilon^\star]$,  we have    
\begin{align*}
& c_s(A_\varepsilon/\varepsilon) \gamma 
\\ &
+ \kappa(\gamma) \left( \frac{c_s(A_\varepsilon/\varepsilon)}{r_s(A_\varepsilon/\varepsilon) - \mu(A_\varepsilon/\varepsilon)}  + \frac{c_u(A_\varepsilon/\varepsilon)}{r_u(A_\varepsilon/\varepsilon) + \mu(A_\varepsilon/\varepsilon)} \right) \Delta 
\\ &  \leq \Delta.  
\end{align*}
Hence,  using \lemstabmani,  Items (i) and (ii) in Lemma \ref{lem6} follow with $U := B_\gamma$ and $r := \gamma/2$.
  
To prove Item (iii),  we apply Tikhonov Theorem---see \thmtykho\ and \RemThyk---to conclude that,  given $T>0$,  we can find  $\rho^*>0$ and $\varepsilon^\ast>0$ sufficiently small and $M>0$ such that, for each $\varepsilon \in [0, \varepsilon^\ast]$,  the solution  
$\bar{x}(t)$ to \eqref{singpertsyst} and the signal $\bar{x}_o(t) := (x_{mo}(t), 0)$,  with $x_{mo}(t)$ solution to \eqref{redsyst},   satisfy
\begin{align*} 
| \bar{x}(0) - \bar{x}_{o}(0) | \leq \rho^* 
\Longrightarrow
| \bar{x}(t) - \bar{x}_{o}(t) | \leq M \rho^* \quad \forall t \in [0, T].
\end{align*}
Now,  we pick $T>0$ as the largest time a solution $x_{m}(t)$ to \eqref{redsyst} starting from $\Gamma_1 \backslash B_r$ takes to enter the ball $B_{r/2}$.  Such a $T>0$ always exists by definition of the set $\Gamma_1$ and is finite since the set $\Gamma_1$ is compact.  By taking $\rho^* = \frac{r}{3M}$, we conclude that 
$|\bar{x}(T)|$ must be inside $B_r$. 


\bibliographystyle{ieeetr}


\begin{thebibliography}{10}

\bibitem{chow1982}
E.~J.~H.~Chow, {\em Time-Scale Modeling of Dynamic Networks with Applications
  to Power Systems}.
\newblock No.~46 in Lecture Notes in Control and Information Sciences,
  Heidelberg: Springer Verlag, 1st.~ed., 1982.

\bibitem{heylighen1989self}
F.~Heylighen, ``Self-organization, emergence and the architecture of
  complexity,'' in {\em Proc. 1st Europ. Conf. Syst. Sc.}, vol.~18, pp.~23--32,
  1989.

\bibitem{Silberstein2002-SILREA}
M.~Silberstein, ``Reduction, emergence and explanation,'' in {\em The Blackwell
  Guide to the Philosophy of Science} (P.~K. Machamer and M.~Silberstein,
  eds.), pp.~80--107, Cambridge: Blackwell, 2002.

\bibitem{WEIBOOK}
W.~Ren and R.~W. Beard, {\em Distributed consensus in multi-vehicle cooperative
  control}.
\newblock London, U.K.: Springer verlag, 2008.

\bibitem{Arecchi1997}
F.~T. Arecchi, {\em A Critical Approach to Complexity and Self Organization},
  ch.~in {Mathematical Undecidability, Quantum Nonlocality and the Question of
  the Existence of God}, pp.~59--81.
\newblock A. Driessen and A. Su\'arez, Eds., Dordrecht, Netherlands: Springer,
  1997.

\bibitem{arcak_TAC2007_passive-networks_4287131}
M.~Arcak, ``Passivity as a design tool for group coordination,'' {\em IEEE
  Trans. Automat. Contr.}, vol.~52, no.~8, pp.~1380--1390, 2007.

\bibitem{isidori-marconi2014_6819823}
A.~Isidori, L.~Marconi, and G.~Casadei, ``Robust output synchronization of a
  network of heterogeneous nonlinear agents via nonlinear regulation theory,''
  {\em IEEE Trans. Automat. Contr.}, vol.~59, no.~10, pp.~2680--2691, 2014.

\bibitem{MoreauCDC04}
L.~Moreau, ``Stability of continuous-time distributed consensus algorithms,''
  {\em 43rd IEEE Conf. Dec. Control}, vol.~4, pp.~3998 --4003, 2004.

\bibitem{chowdhury2016persistence}
N.~R. Chowdhury, S.~Sukumar, and N.~Balachandran, ``Persistence-based
  convergence rate analysis of consensus protocols for dynamic graph
  networks,'' {\em Europ. J. Contr.}, vol.~29, pp.~33--43, 2016.

\bibitem{olfati2004consensus}
R.~R. Olfati-Saber and R.~M. Murray, ``Consensus problems in networks of agents
  with switching topology and time-delays,'' {\em IEEE Trans. Automat. Contr.},
  vol.~49, no.~9, pp.~1520--1533, 2004.

\bibitem{Morarescu-switching_9105084}
B.~Adhikari, I.-C. Mor{\u{a}}rescu, and E.~Panteley, ``An emerging dynamics
  approach for synchronization of linear heterogeneous agents interconnected
  over switching topologies,'' {\em IEEE Contr. Syst. Lett.}, vol.~5, no.~1,
  pp.~43--48, 2021.

\bibitem{al:SCLCHANNELS}
A.~Teel, A.~Lor\'{\i}a{}, E.~Panteley, and D.~Popovi\'c, ``Smooth time-varying
  stabilization of driftless systems over communication channels,'' {\em Syst.
  Contr. Lett.}, vol.~55, no.~12, pp.~982--991, 2006.

\bibitem{casadei2019_8713390}
G.~Casadei, D.~Astolfi, A.~Alessandri, and L.~Zaccarian, ``Synchronization in
  networks of identical nonlinear systems via dynamic dead zones,'' {\em IEEE
  Contr. Syst. Lett.}, vol.~3, no.~3, pp.~667--672, 2019.

\bibitem{martin2016time}
S.~Martin, I.-C. Mor{\u{a}}rescu, and D.~Ne{\v{s}}i{\'c}, ``Time scale modeling
  for consensus in sparse directed networks with time-varying topologies,'' in
  {\em IEEE 55th Conf. Dec. Contr.}, pp.~7--12, IEEE, 2016.

\bibitem{altafini2012_6329411}
C.~Altafini, ``Consensus problems on networks with antagonistic interactions,''
  {\em IEEE Trans. Automat. Contr.}, vol.~58, no.~4, pp.~935--946, 2013.

\bibitem{pogromsky2011_6160437}
A.~Pogromsky, N.~Kuznetsov, and G.~Leonov, ``Pattern generation in diffusive
  networks{: How} do those brainless centipedes walk?,'' in {\em Proc. 50th
  IEEE Conf. Dec. Contr. and Europ. Contr. Conf.}, pp.~7849--7854, 2011.

\bibitem{PANTELEY_SCHOLL}
L.~Tumash, E.~Panteley, A.~Zakharova, and E.~Scholl, ``Synchronization patterns
  in {Stuart-Landau} networks: a reduced system approach,'' {\em The European
  Physical Journal B -- Condensed Matter and Complex Systems}, vol.~92, no.~5,
  2019.

\bibitem{DYNCON-TAC16}
E.~Panteley and A.~Lor\'{\i}a, ``Synchronization and dynamic consensus of
  heterogeneous networked systems,'' {\em IEEE Trans. Automat. Contr.},
  vol.~62, no.~8, pp.~3758--3773, 2017.

\bibitem{shim_AUT2020}
J.~G. Lee and H.~Shim, ``A tool for analysis and synthesis of heterogeneous
  multi-agent systems under rank-deficient coupling,'' {\em Automatica},
  vol.~117, p.~108952, 2020.

\bibitem{HDR-LENA}
E.~Panteley, ``A stability-theory perspective to synchronisation of
  heterogeneous networks.'' Habilitation \`a diriger des recherches (DrSc
  dissertation). Universit{\'e} Paris Sud, Orsay, France, 2015.
\newblock Available online: https://hal.archives-ouvertes.fr/tel-01262772/.

\bibitem{CP3-CDC16-Mohamed}
M.~Maghenem, E.~Panteley, and A.~Lor{\'\i}a, ``Singular-perturbations-based
  analysis of synchronization in heterogeneous networks: a case-study,'' in
  {\em {Proc. 55th IEEE Conf. Dec. Contr.}}, ({Las Vegas, NV, USA}),
  pp.~2581--2586, 2016.

\bibitem{JP2-STUART-LANDAU}
E.~Panteley, A.~Lor\'{\i}a, and A.~E{l-Ati}, ``Practical dynamic consensus of
  {Stuart-Landau} oscillators over heterogeneous networks,'' {\em Int. J.
  Control}, vol.~93, no.~2, pp.~261--273, 2020.

\bibitem{wieland-sep-AUT-2011}
P.~Wieland, R.~Sepulchre, and F.~Allg\"ower, ``An internal model principle is
  necessary and sufficient for linear output synchronization,'' {\em
  Automatica}, vol.~47, no.~5, pp.~1068--1074, 2011.

\bibitem{depersis2014_TCNS_internal-model}
C.~de~Persis and B.~Jayawardhana, ``On the internal model principle in the
  coordination of nonlinear systems,'' {\em IEEE Trans. Contr. Net. Syst.},
  vol.~1, no.~3, pp.~272--282, 2014.

\bibitem{kim2011_05605658}
H.~Kim, H.~Shim, and J.~H. Seo, ``Output consensus of heterogeneous uncertain
  linear multi-agent systems,'' {\em IEEE Trans. Automat. Contr.}, vol.~56,
  no.~1, pp.~200--206, 2011.

\bibitem{haken77synergetics}
H.~Haken, {\em Synergetics}.
\newblock Berlin, Heidelberg, New York: Springer-Verlag, 1977.

\bibitem{anosov1960limit}
D.~V. Anosov, ``Limit cycles of systems of differential equations with small
  parameters in the highest derivatives,'' in {\em Eleven papers on analysis},
  vol.~92, pp.~299--334, Translation by the American Mathematical Society,
  Russian Academy of Sciences, Branch of Mathematical Sciences, 1963.

\bibitem{kokotovic1976singular}
P.~V. Kokotovi\'c, R.~E. O'Malley, and P.~Sannuti, ``Singular perturbations and
  order reduction in control theory—an overview,'' {\em Automatica}, vol.~12,
  no.~2, pp.~123--132, 1976.

\bibitem{kokotovic1999singular}
P.~V. Kokotovi{\'c}, H.~Khalil, and {J.~O'}Reilly, {\em Singular perturbation
  methods in control: analysis and design}.
\newblock SIAM, 1999.

\bibitem{KHALIL96}
H.~Khalil, {\em Nonlinear systems}.
\newblock New York: {Macmillan Publishing Co., 2nd ed.}, 1996.

\bibitem{chow_kokotovic1985}
J.~Chow and P.~Kokotovi\'c, ``Time scale modeling of sparse dynamic networks,''
  {\em IEEE Trans. Automat. Contr.}, vol.~30, no.~8, pp.~714--722, 1985.

\bibitem{biyik2008area}
E.~B{\i}y{\i}k and M.~Arcak, ``Area aggregation and time-scale modeling for
  sparse nonlinear networks,'' {\em Syst. Contr. Lett.}, vol.~57, no.~2,
  pp.~142--149, 2008.

\bibitem{tognetti2021synchronization}
E.~S. Tognetti, T.~R. Calliero, I.-C. Mor{\u{a}}rescu, and J.~Daafouz,
  ``Synchronization via output feedback for multi-agent singularly perturbed
  systems with guaranteed cost,'' {\em Automatica}, vol.~128, p.~109549, 2021.

\bibitem{rejeb2018control}
J.~B. Rejeb, I.-C. Mor{\u{a}}rescu, and J.~Daafouz, ``Control design with
  guaranteed cost for synchronization in networks of linear singularly
  perturbed systems,'' {\em Automatica}, vol.~91, pp.~89--97, 2018.

\bibitem{montenbruck2015practical}
J.~M. Montenbruck, M.~B{\"u}rger, and F.~Allg{\"o}wer, ``Practical
  synchronization with diffusive couplings,'' {\em Automatica}, vol.~53,
  pp.~235--243, 2015.

\bibitem{hill2015}
T.~Liu, D.~Hill, and J.~Zhao, ``Output synchronization of dynamical networks
  with incrementally-dissipative nodes and switching topology,'' {\em IEEE
  Trans. Circ. Syst. I: Fundam. Th. Appl.}, vol.~62, no.~9, pp.~2312--2323,
  2015.

\bibitem{delellis2011quad}
P.~DeLellis, M.~D. Bernardo, and G.~Russo, ``On quad, lipschitz, and
  contracting vector fields for consensus and synchronization of networks,''
  {\em IEEE Trans. on Circuits and Systems I:}, vol.~58, no.~3, pp.~576--583,
  2011.

\bibitem{teel1999semi}
A.~A.~R.~Teel, J.~Peuteman, and D.~Aeyels, ``Semi-global practical asymptotic
  stability and averaging,'' {\em Syst. Contr. Lett.}, vol.~37, no.~5,
  pp.~329--334, 1999.

\bibitem{HALE_BOOK}
J.~K. Hale, {\em Ordinary Differential equations}, vol.~21 of {\em
  Interscience}.
\newblock {New York}: {John Wiley}, 1969.

\bibitem{FLOQUET}
G.~Floquet, ``Sur les \'quations diff\'rentielles lin\'aires \`a coefficients
  p\'eriodiques,'' {\em {Annales de l'\'Ecole Normale Sup\'erieure}}, no.~12,
  pp.~47--88, 1883.

\bibitem{Perko}
L.~Perko, {\em Differential Equations and Dynamical Systems}.
\newblock Springer, 2000.

\bibitem{BASENER2006841}
W.~Basener, B.~P. Brooks, and D.~Ross, ``The brouwer fixed point theorem
  applied to rumour transmission,'' {\em Appl. Math. Lett.}, vol.~19, no.~8,
  pp.~841--842, 2006.

\bibitem{tikhonov}
A.~N. Tikhonov, ``Systems of differential equations containing small parameters
  in the derivatives,'' {\em Matematicheskii Sbornik}, vol.~73, no.~3,
  pp.~575--586, 1952.

\bibitem{moro1997lidskii}
J.~Moro, J.~V. Burke, and M.~L. Overton, ``On the lidskii--vishik--lyusternik
  perturbation theory for eigenvalues of matrices with arbitrary jordan
  structure,'' {\em SIAM J. Matrix Anal. \& Appl.}, vol.~18, no.~4,
  pp.~793--817, 1997.

\bibitem{Stewart90}
G.~W. Stewart and G.-J. Sun, {\em Matrix Perturbation Theory}.
\newblock Academic Press, 1990.

\bibitem{coddington1955theory}
E.~A. Coddington and N.~Levinson, {\em Theory of ordinary differential
  equations}.
\newblock McGraw Hill, 1955.

\bibitem{Kuznetsov}
Y.~A. Kuznetsov, {\em Elements of Applied Bifurcation Theory}.
\newblock Springer, Applied Mathematical Sciences, Vol. 112, 1998.

\bibitem{Strogatz2}
P.~C. Matthews and S.~H. Strogatz, ``Phase diagram for the collective behavior
  of limit-cycle oscillators,'' {\em Phys. Rev. Lett.}, vol.~65,
  pp.~1701--1704, Oct 1990.

\bibitem{slot}
Q.~C. Pham and J.~J. Slotine, ``Stable concurrent synchronization in dynamic
  system networks,'' {\em {Neural Networks}}, vol.~20, no.~1, pp.~62--77, 2007.

\bibitem{MICNON-2015}
E.~Panteley, A.~Lor\'{\i}a{}, and A.~{El A}ti, ``On the stability and
  robustness of {Stuart-Landau} oscillators,'' {\em IFAC-PapersOnLine},
  vol.~48, no.~11, pp.~645--650, 2015.
\newblock Presented at the 1st IFAC Conference onModelling, Identification
  andControl of Nonlinear Systems, MICNON 2015, {\em (St. Petersburg, Russia)}.

\bibitem{angeli2015characterizations}
D.~Angeli and D.~Efimov, ``Characterizations of input-to-state stability for
  systems with multiple invariant sets,'' {\em IEEE Trans. Automat. Contr.},
  vol.~60, no.~12, pp.~3242--3256, 2015.

\bibitem{GOHLANROD}
I.~Gohberg, P.~Lancaster, and L.~Rodman, {\em Invariant Subspaces of Matrices
  with Applications}.
\newblock Society for Industrial and Applied Mathematics, 2006.

\bibitem{ye2016schubert}
K.~Ye and L.-H. Lim, ``Schubert varieties and distances between subspaces of
  different dimensions,'' {\em SIAM J. Matrix Anal. \& Appl.}, vol.~37, no.~3,
  pp.~1176--1197, 2016.

\bibitem{LectureA}
B.~Birnir, ``Dynamical systems theory.'' University Lecture, 2008.

\bibitem{Stone1948}
M.~H. Stone, ``The generalized weierstrass approximation theorem,'' {\em
  Mathematics Magazine}, vol.~21, no.~4, pp.~167--184, 1948.

\bibitem{hartman2002ordinary}
P.~Hartman, {\em Ordinary differential equations}.
\newblock SIAM, 2002.

\end{thebibliography}
 
\end{document}